\documentclass[11pt]{amsart}
\usepackage{url}
\usepackage{hyperref}
\usepackage[hyphenbreaks]{breakurl}
\usepackage{subfiles}
\usepackage{tikz-cd}
\usepackage{geometry} 
\usepackage{verbatim} 
\usepackage{color}
\usepackage{amsmath}
\usepackage{amsfonts, amssymb, amsthm, setspace, latexsym}
\usepackage{multicol}
\usepackage{graphicx}
\usepackage[all,cmtip,curve, knot,frame]{xy} 
\usepackage[at]{easylist}
\usepackage{subfiles}
\usepackage{enumitem}
\usepackage{setspace}
\usepackage{todonotes}
\usepackage{scalefnt}
\usepackage{tikz}

\usepackage{xcolor}
\hypersetup{
    colorlinks,
    linkcolor={red!50!black},
    citecolor={blue!50!black},
    urlcolor={blue!80!black}
}

\usetikzlibrary{knots}

\newcommand{\mvhide}[1]{}

\newcommand{\Y}{\mathcal{Y}}
\newcommand{\KY}{\K[\Y]}
\newcommand{\RY}{\cR[\Y]}
\newcommand{\RYn}{\cR[\Y_n]}

\newcommand{\X}{\mathcal{X}}
\newcommand{\HKuniv}{\HKop}

\newcommand{\HCq}{\mathrm{HC}_\q}

\newcommand{\M}{M}

\newcommand{\ellen}{\ell en}
\newcommand{\ellgen}{\ell} 

\newcommand{\q}{\mathtt{q}}
\newcommand{\Sq}{\boldsymbol{q}} 
\newcommand{\Spi}{\overline{\pi}} 
\newcommand{\Q}{\mathbb{Q}} 
\newcommand{\K}{\mathcal{K}}

\newcommand{\AffSym}{\widehat{\mathfrak{S}}_n}
\newcommand{\AffSymn}[1]{\widehat{\mathfrak{S}}_{#1}}

\newcommand{\Sym}{{\mathfrak{S}}}
\newcommand{\Zprod}{\boldsymbol{Z}} 
\newcommand{\Sn}{{\Sym}_n}
\newcommand{\SN}{{\Sym}_{N}}
\newcommand{\Sm}[1]{{\Sym}_{{#1}}}
\newcommand{\id}{\operatorname{Id}}

\newcommand{\uu}{\underline{u}}

\newcommand{\Bullet}{\begin{tikzpicture}\draw[fill=gray] (0,0) circle (.7ex); \end{tikzpicture} }

\newcommand{\cA}{\mathcal{A}}

\newcommand{\cO}{\mathcal{O}}

\newcommand{\SkCat}{\mathrm{SkCat}}

\newcommand{\SkAlg}{\mathrm{SkAlg}}

\newcommand{\DAHA}[2]{\mathrm{\mathbb{H}_{#2}}}

\newcommand{\iBt}{i_{\widetilde{\mathcal B}}}
\newcommand{\DN}{\HH_N} 
\newcommand{\SymY}{S(\Y)}
\newcommand{\SymYn}{S(\Y_n)}
\newcommand{\Yn}{\Y_n}
\DeclareMathOperator{\Stab}{Stab}

\newcommand{\cuspell}[1]{\operatorname{Cusp}_{\operatorname{ell}}({#1})}

\newcommand{\SWdaff}{\ddot{\mathrm{SW}}}
\newcommand{\SWe}{\SWdaff \ee}  
\newcommand{\SWeN}{\SWdaff \eN}
\newcommand{\eSWe}{\ee \SWdaff \ee}  
\newcommand{\eSWeN}{\eN \SWdaff \eN}

\newcommand{\syt}{\mathrm{SYT}}

\newcommand{\bt}{\boxtimes}
\DeclareMathOperator{\Comm}{Comm}

\newcommand{\ZZ}{\mathbb{Z}}

\newcommand{\cD}{\mathcal{D}}
\newcommand{\Dq}{{\cD_\q(G)}}
\newcommand{\Oq}{\cO_\q(G)}
\newcommand{\DqG}[1]{{\cD_\q(#1)}}
\newcommand{\OqG}[1]{\cO_\q(#1)}
\newcommand{\DqqG}[1]{{\cD_q(#1)}}

\newcommand{\Uq}{U_\q(\g)}
\newcommand{\Uqgl}{U_\q(\glN)}
\newcommand{\Uqsl}{U_\q(\slN)}
\newcommand{\detq}{\operatorname{det}_\q}

\newcommand{\cR}{\mathcal{R}}
\newcommand{\cU}{U}
\newcommand{\g}{\mathfrak{g}}
\newcommand{\HH}{\mathbb{H}}
\newcommand{\HG}{\HH_n^{\GL}({q,t})}
\newcommand{\HS}{\HH_n^{\SL}(\Sq,t)}
\newcommand{\HGN}{\HH_N^{\GL}({q,t})}

\renewcommand{\H}{\operatorname{H}} 
\newcommand{\Hf}{\operatorname{H}^{\mathrm{fin}}}  
\newcommand{\Hfn}{\Hf_n}
\newcommand{\SH}{{\mathcal{S}H}}  
\newcommand{\SHn}{\mathcal{S}H_n} 

\newcommand{\C}{\mathbb{C}}
\newcommand{\tr}{\mathrm{tr}}
\newcommand{\Rep}{\operatorname{Rep}}

\newcommand{\ot}{\otimes}
\newcommand{\rt}[1]{\underset{#1}{\otimes}}

\newcommand{\modu}{\textrm{-mod}}
\newcommand{\Hom}{\operatorname{Hom}}
\newcommand{\SL}{\mathrm{SL}}
\newcommand{\SLN}{\SL_N}
\newcommand{\SLn}{\SL_n}
\newcommand{\GL}{\mathrm{GL}}
\newcommand{\GLN}{\GL_N}

\newcommand{\slN}{\mathfrak{sl}_N}
\newcommand{\glN}{\mathfrak{gl}_N}
\newcommand{\F}{F'}
\DeclareMathOperator{\Res}{Res}
\DeclareMathOperator{\Ind}{Ind} 
\DeclareMathOperator{\End}{End} 

\newcommand{\HC}{\operatorname{HC}}
\newcommand{\HKe}{\HKop(\epsilon)}
\newcommand{\HKop}{\operatorname{HK}}

\newcommand{\HKchi}{\HKop(\chi)}

\newcommand{\ad}{\operatorname{ad}}
\newcommand{\un}{\mathbf{1}}
\newcommand{\Dist}{\operatorname{Dist}}

\newcommand{\Repq}{\mathrm{Rep}_{\q}}

\newcommand{\sgn}{\operatorname{sgn}}
\newcommand{\asgn}{ \{ \underline{{\mathbf a}}\} \boxtimes \sgn}
\newcommand{\atup}{{\underline{{\mathbf a}} }}
\newcommand{\btup}{{\underline{{\mathbf b}} }}
\newcommand{\lambdatup}{{\underline{{\mathbf \lambda}} }}
\newcommand{\vv}{\underline{v}}  
\newcommand{\ek}[1]{\mathrm{e}_{{#1}}} 
\newcommand{\eN}{\ek{N}}
\newcommand{\en}{\ek{n}}
\newcommand{\ee}{\mathrm{e}} 
\newcommand{\etrivn}{\etriv{n}}
\newcommand{\etriv}[1]{\mathrm{e}^{+}_{#1}}
\newcommand{\Coinvq}{\cD_q({H})^{\SN}} 

\newcommand{\SmashDq}{\cD_{q}({H}) \# \SN}
\newcommand{\SmashY}{\K[\Y] \# \SN}
\newcommand{\coinv}{\mathtt{S}}  
\newcommand{\smashD}{\mathtt{D}}
\newcommand{\smashY}{\Gamma}

\newcommand{\SymSN}{\SymS_N}
\newcommand{\SymS}{\mathtt{S} \mathfrak{S}}
\newcommand{\cas}[1]{c_{#1}}  

\newcommand{\cl}{\operatorname{\mathtt{cl}}} 
\DeclareMathOperator{\spec}{Spec}
\DeclareMathOperator{\MatN}{Mat_N}

\DeclareMathOperator{\trivial}{\operatorname{triv}}
\DeclareMathOperator{\Rosso}{\mathfrak{R}} 

\newcommand{\Dqstr}[1]{\mathcal{D}_\q({#1}){\modu}^{\hspace{0.5pt} {#1}}}
\newcommand{\Dqadm}[1]{\mathcal{D}_\q({#1}){\modu}^{\, \mathrm{adm}}}
\newcommand{\DqadmC}[1]{\mathcal{D}_\q({#1}){\modu}^{\, \mathrm{adm}}_C}

\newcommand{\Vect}{\operatorname{Vect}}

\newcommand{\fg}{\mathfrak{g}}
\newcommand{\fh}{\mathfrak{h}}
\newcommand{\ft}{\mathfrak{t}}

\newcommand{\cB}{\mathcal{B}}

\DeclareMathOperator{\Spec}{Spec}

\newcommand{\aff}{\operatorname{\textit{aff}}}

\newcommand{\GIT}{{/\! /}}

\newcommand{\op}{\mathrm{op}}

\newcommand{\coker}{\mathrm{coker}}

\newcommand{\phii}{\varphi_i} 
\newcommand{\phij}[1]{\varphi_{{#1}}}
\newcommand{\nui}{\nu_i} 
\newcommand{\nuj}[1]{\nu_{{#1}}}
\newcommand{\deltai}{\delta_i}
\newcommand{\aij}[2]{a_{{#1, #2}}} 
\newcommand{\bij}[2]{b_{{#1, #2}}} 
\newcommand{\fij}[2]{f_{{#1, #2}}} 
\newcommand{\qrho}{q^{\rho}} 
\newcommand{\trho}{t^{-2\rho}} 
\newcommand{\HHloc}{\widetilde{\HH}}
\newcommand{\Yloc}{\widetilde{\Y}}
\newcommand{\normal}[1]{\colon\! {#1} \colon\!} 
\newcommand{\Drho}{\Ind^\HH_\Y(\trho)} 
\newcommand{\HX}{\H(\cX)}
\newcommand{\HY}{\H(\Y)}
\newcommand{\HYn}{\H_n(\Y)}
\newcommand{\trrho}{\gamma}

\newcommand{\Inv}{\operatorname{Inv}}

\renewcommand{\r}{\mathfrak{r}}

\newcommand{\cC}{\mathcal C}

\newcommand{\cK}{\mathcal K}
\newcommand{\cX}{\mathcal X}
\newcommand{\cY}{\mathcal Y}

\newcommand{\Z}{\mathbb Z}

\makeatletter
\newcommand*\leftdash{\rotatebox[origin=c]{-45}{$\dabar@\dabar@\dabar@$}}
\newcommand*\rightdash{\rotatebox[origin=c]{45}{$\dabar@\dabar@\dabar@$}}
\makeatother

\newcommand{\wedgeN}[1]{\Lambda^{#1} V} 

\newcommand{\omitt}[1]{ }
\newcommand{\Rform}[1]{#1_{\cR}}  
\newcommand{\Kform}[1]{#1_{\cK}}  
\newcommand{\qform}[1]{#1_{{\kappa}}}  

\tikzstyle{V}=[draw, fill =black, circle, inner sep=0pt, minimum size=1.
5pt]
\tikzstyle{bV}=[draw, fill =black, circle, inner sep=0pt, minimum size=4
pt]
\tikzstyle{over}=[draw=white,double=black, double distance=1.75pt]
\tikzstyle{diagram}=[line width=.75pt, scale=\SCALE]

\def\Over[#1,#2][#3,#4]{ 
        \draw[over]   (#1,#2) .. controls ++(0,#4*.5-#2*.5) and ++(0,-#4*.5+#2*.5) .. (#3,#4);}
\def\Under[#1,#2][#3,#4]{ 
        \draw  (#1,#2) .. controls ++(0,#4*.5-#2*.5) and ++(0,-#4*.5+#2*.5) .. (#3,#4);}
\def\Cross[#1,#2][#3,#4]{
        \Under[#3,#2][#1,#4]\Over[#1,#2][#3,#4]}

\numberwithin{equation}{section}
\newtheorem*{theorem*}{Theorem}
\newtheorem{theorem}{Theorem}[section]
\newtheorem{theorem/def}[theorem]{Theorem/Definition}
\newtheorem{lemma}[theorem]{Lemma}
\newtheorem{proposition}[theorem]{Proposition}
\newtheorem{conjecture}[theorem]{Conjecture}
\newtheorem{corollary}[theorem]{Corollary}
\theoremstyle{definition}
\newtheorem{definition}[theorem]{Definition}
\newtheorem{remark}[theorem]{Remark}
\newtheorem{notation}[theorem]{Notation}
\newtheorem{example}[theorem]{Example}
\newtheorem{def/prop}[theorem]{Definition/Proposition}

\newcommand{\defterm}[1]{\textbf{#1}}

    \newcommand{\leftquotient}[2]{\raisebox{\depth}{$#1$}\Big/ \raisebox{-\depth}{$#2$} }
    
 \newcommand{\rightquotient}[2]{{\raisebox{-\depth}{$#1$}{\Big \backslash} \raisebox{\depth}{$#2$}}}
 \newcommand{\doublequotient}[3]{\raisebox{-\depth}{{$#1$}}\Big \backslash \raisebox{\depth}{{$#2$}}\Big/ \raisebox{-\depth}{{$#3$}}}

    \newcommand{\Cchi}{C(\chi)}
        \newcommand{\SLGL}{({\color{blue}\diamond})} 

\usepackage[style=alphabetic, minnames=4, maxnames=4,maxalphanames=4, minalphanames=4]{biblatex}
\title[Quantum character theory]{Quantum character theory}
\author{Sam Gunningham, David Jordan and Monica Vazirani}
\date{\today}
\calclayout
\begin{document}
\maketitle

\begin{abstract}
   We develop a $\q$-analogue of the theory of conjugation equivariant $\cD$-modules on a complex reductive group $G$. In particular, we define  quantum Hotta-Kashiwara modules and compute their endomorphism algebras.  We use the Schur-Weyl functor of the second author, and develop tools from the corresponding double affine Hecke algebra to study this category in the cases $G=\GLN$ and $\SLN$.   Our results also have an interpretation in skein theory (explored further in a sequel paper), namely a computation of the $\GLN$ and $\SLN$-skein algebra of the 2-torus. 
\end{abstract}

\tableofcontents

\section{Introduction}

\subsection{Classical character theory}
Conjugation equivariant $\cD$-modules on a reductive group and on its Lie algebra have a rich history in representation theory. Harish-Chandra observed that the distributional characters of irreducible unitary representations of semisimple groups satisfy a holonomic system of differential equations, which he used to initiate the character theory of unitary representations \cite{HarishChandra1965}. Hotta and Kashiwara defined and studied the corresponding $\cD$-module, revealing a striking connection with Springer theory -- the study of representations of the Weyl group on certain cohomology groups of Springer fibers \cite{HK}. These ideas have since come to be understood in a wider context including Lusztig's generalized Springer correspondence and his theory of character sheaves \cite{Lusztig1984,Lusztig1985}, as interpreted in the $\cD$-module setting by Ginzburg \cite{Ginsburg1989}, and further developed in a categorical context by Ben-Zvi and Nadler \cite{BZN} and the first author \cite{G}.

Let us recall some of the basic objects from this theory. Given a connected complex reductive group $G$ with Lie algebra $\fg$, maximal torus $H$ and corresponding Cartan subalgebra $\fg$, we have the ring of algebraic differential operators $\cD(\fg)$ and the category of strongly equivariant modules (for the adjoint action) $\cD(\fg)\modu^{G}$.
The prototypical example of a strongly equivariant $\cD(\fg)$-module is the \emph{universal Hotta-Kashiwara module},
\[
\HKuniv^{\cl} = \cD(\fg)/\cD(\fg)\ad(\fg).
\]
The endomorphism algebra of $\HKuniv^{\cl}$ is naturally identified with its $G$-invariants; this is known as the quantum Hamiltonian reduction:
\[
\cD(\fg)\GIT G = \End(\HKuniv^{\cl}) \cong \left(\HKuniv^{\cl}\right)^{G}.
\]
A well-known result of Levasseur and Stafford \cite{LS95,LS96} provides an isomorphism
$
\cD(\fg)\GIT G \cong \cD(\fh)^W.
$

Note that we have an embedding $\C[\fh^\ast]^W \cong \operatorname{Sym}(\fg)^G \hookrightarrow \cD(\fg)$. A strongly equivariant $\cD(\fg)$-module is called \emph{admissible} if this subalgebra acts locally finitely (see \cite{Ginsburg1989}). The primary 
examples of admissible modules are the Hotta-Kashiwara modules associated to a fixed character $\zeta \in \fh/W$:
\[
\HKop^{\cl}(\zeta) = \raisebox{\depth}{$\cD(\mathfrak{g})$} \Bigg/ \raisebox{-\depth}{$\left( 
    \cD(\mathfrak{g})\ad(\fg) +  \sum_{y\in \C[\mathfrak{g}]^G} \cD(\mathfrak{g})(y-\zeta(y))\right)$}.
\]
One of the main results of \cite{HK} is to identify the endomorphism algebra of these modules with the group algebra of the subgroup $W_\zeta$ of the Weyl group $W$ corresponding to the stabilizer of a lift of $\zeta$ to $\fh^\ast$, giving a $\cD$-module interpretation of the Springer representations. 

\subsection{Quantum character theory} 

The goal of this paper is to $\q$-deform these notions using the representation theory of quantum groups. 
Let $\K$ denote a ground field containing the quantum parameter $\q$, which we assume is not a root of unity (see Notation \ref{not:groundfieldquantum}). 

Given a reductive group $G$ with Lie algebra $\fg$, maximal torus $H$ and corresponding Cartan subalgebra $\fg$, we have the ring of quantum differential operators $\Dq$ and the category of strongly equivariant $\Dq$ modules (for the quantum adjoint action) $\Dqstr{G}$, determined by a quantum moment map $\ad:\Oq\to\Dq$ constructed in \cite{VV}, and further studied in \cite{J2011},\cite{BZBJ2},\cite{Safronov-qmm-paper}.
The prototypical example of a strongly equivariant $\Dq$-module is the \emph{universal Hotta-Kashiwara module},
\[
\HKuniv = \Dq/\Dq C(\ad),
\]
where $C(\ad)=\{\ad(\ell)-\epsilon(\ell) \mid \ell \in \Oq\}$ and $\epsilon:\Oq \to \K$ is the counit. As in the classical case, the endomorphism algebra of $\HKuniv$ is naturally identified with its $G$-invariants; this is known as the quantum multiplicative Hamiltonian reduction:

\[
\Dq\GIT G = \End(\HKuniv) \cong \left(\HKuniv\right)^{G}.
\]

We have an embedding $\mathcal{O}(H)^W \cong \Oq^G \hookrightarrow \Dq$. A strongly equivariant $\Dq$-module is called \emph{admissible} (or a \emph{$\q$-character sheaf}) if this subalgebra acts locally finitely. The primary 
examples of admissible modules are the Hotta-Kashiwara modules associated to a fixed character $\chi \in H/W$:
\begin{equation}\label{eqn:HKchi}
\HKchi = \raisebox{\depth}{$\Dq$}\Bigg/\raisebox{-\depth}{$\left( \Dq C(\ad) + \sum_{y\in\Oq^G} \Dq(y-\chi(y))\right)$}
\end{equation}

We will typically focus our attention on reductive groups of type $A$, specifically $G= \SLN$ and $G= \GLN$. This is because our main technical tool is the Schur-Weyl functor $F_N$ defined by the second author \cite{J2008}, which relates strongly equivariant $\Dq$-modules with modules for a certain specialization of the double affine Hecke algebra (DAHA) $\DAHA{N}{N}$ (see Definition \ref{def:GLDAHA}). Our first main result computes the endomorphism algebra of $\HKuniv$ in these cases.

\begin{theorem}\label{mainthm:HKuniv}
Let $G=\GLN$ or $\SLN$.  We have a canonical isomorphism,
\[
\End(\HKuniv)  \cong
\ee \cdot \DAHA{N}{N}^{\op} \cdot \ee
\]
\end{theorem}

Here, $\ee \in \DAHA{N}{N}$ is the sign idempotent. We refer to $\ee \cdot \DAHA{N}{N}^{\op} \cdot \ee$ as the antispherical DAHA.  As a special case of the \defterm{shift isomorphism} (see Section \ref{sec:shift}), we can identify the righthand side of Theorem \ref{mainthm:HKuniv} with the algebra $\DqG{H}^{\SN}$ of $\SN$-invariant differential operators, and hence we interpret Theorem \ref{mainthm:HKuniv} as giving a quantum Levasseur-Stafford isomorphism:
\begin{equation}\label{eqn:LS}
\Dq/\!/G \cong \cD_\q(H)^{\SN}.    
\end{equation}

Recall that a key result of Hotta and Kashiwara was an identification of $\HKuniv^{\cl}(0)$ with the pushforward of the constant sheaf along the Springer resolution, implying in particular an isomorphism between its endomorphism algebra and the group algebra of the Weyl group.  The quantum analogue of $\HKuniv^{\cl}(0)$ is $\HKe$, where $\epsilon: \Oq\to \K$ denotes the counit character, corresponding classically to the identity element in $G$.  Our second main result may be regarded as a $\q$-deformed Springer correspondence \`a la Hotta--Kashiwara.

\begin{theorem}\label{mainthm:Springer}
	Let $G=\GLN$ or $\SLN$. We have a canonical isomorphism
	\[
	\End(\HKe) \cong \K[\SN]^{\op}.
	\]
 \end{theorem}

\begin{theorem}\label{mainthm:decomp}\
Consider the resulting decomposition
\begin{align*} 
\HKe = \bigoplus_{\lambda} M_\lambda \boxtimes S^\lambda,
\end{align*}
as $\Dq$-$\K[\SN]$ bimodules
where the direct sum is indexed by partitions $\lambda$ of $N$ and $S^\lambda$ is the corresponding irreducible for $\SN$.  Then the factors $M_\lambda$  appearing are distinct indecomposable unipotent $\q$-character sheaves.
\end{theorem}

 More generally we compute the endomorphism algebra of the modules $\HKchi$, under mild a combinatorial assumption (see Section \ref{sec:End HK via shift} for details) on the character $\chi$. It similarly leads to a decomposition as bimodules.

\begin{theorem}\label{mainthm:Springerchi} 
Let $\chi\in H/W$ and pick a lift $\bar \chi \in H$. 
Then $\bar \chi$ has stabilizer $\sigma W_J \sigma^{-1}$ for some $\sigma \in \AffSym$, where $W_J\subset W= \SN$ is a standard parabolic subgroup.    Suppose that $\chi$ is in transverse position, i.e., that $W \cap \sigma W_J \sigma^{-1} = \{ \id\} $. Then we have an isomorphism,
\[
\End(\HKchi) \cong \K[W_J]^{op}.
\]
 Consider the resulting decomposition
    \begin{align*} 
\HKchi = \bigoplus_{\lambdatup} M_\lambdatup \boxtimes S^\lambdatup,
\end{align*}
as $\Dq$-$W_J$ bimodules,
where the direct sum is indexed by multi-partitions $\lambdatup$  refining the composition corresponding to $J$ and of total size $N$, and   $S^{\lambdatup}$ is the corresponding irreducible for $W_J$.  Then the factors $M_{\lambdatup}$  appearing are distinct indecomposable $\q$-character sheaves.
\end{theorem}

\begin{remark}
Let us emphasize that the category $\Dqstr{G}$ is not semi-simple, and so the semi-simplicity of endomorphism algebras as in Theorems \ref{mainthm:Springer} and \ref{mainthm:Springerchi} is not automatic.  In Example \ref{ex:nilpotent end} we illustrate a $\chi$ not in transverse position, such that  $\End(\HKchi)$ is not semisimple.
\end{remark}

\subsection{Results of \cite{GJVY} pertaining to $\q$-character theory}
In our forthcoming paper \cite{GJVY}, with Haiping Yang, we prove the following results, which we state here for context.\footnote{We have chosen to emphasize the applications to skein theory in \cite{GJVY} and the applications to $\q$-character theory in the present paper.}

\begin{theorem}\label{thm:HK generates (GJV)}
 Let $G=\GLN$. The object $\HKuniv$ is a (compact, projective) generator of $\Dqstr{G}$. In particular, there is an equivalence of categories
    \[
    \Dqstr{G} \simeq \ee \cdot \DN \cdot \ee\modu \simeq \cD_q(H)^W\modu
    \]
\end{theorem}
In fact, we also prove that $\ee \cdot \DN$ defines a Morita equivalence between $\DN$ and $\ee \cdot \DN \cdot \ee$, so one may replace $\ee \cdot \DN \cdot \ee$ by $\DN$ in Theorem \ref{thm:HK generates (GJV)}. Using a similar, well-known Morita equivalence for the symmetric idempotent $\ee_+$, one may replace $\cD_q(H)^{\SN}$ by $\cD_q(H)\# \SN$. 
It is somewhat more involved to state the $\SLN$ analogue of these results. We will give a full statement in \cite{GJVY}, and state here the case of $\SL_{2}$.

\begin{theorem}\label{thm:DqstrSL2}
    Let $G=\SL_2$. There is an equivalence of categories
    \[
    \Dqstr{G} = \cD_q(H)^{\Sm{2}}\modu \oplus \Vect_\K^{\oplus 4}
    \]
\end{theorem}
\begin{remark}
    The four simple objects appearing above may be understood as cuspidal $\q$-character sheaves (see the discussion in Section \ref{sec:quantum springer}). 
\end{remark}

As a consequence of Theorem \ref{thm:HK generates (GJV)}, we are able to show that the modules $M_\lambda$ and $M_{\lambdatup}$ are irreducible, not only indecomposable, and that in the $\GLN$ case they exhaust all irreducible unipotent $\q$-character sheaves.

\subsection{Characters of classical and quantum Harish-Chandra bimodules} 
Replacing $\fg$ with $G$, one may also consider the category $\cD(G)\modu^{G}$ of strongly equivariant $\cD(G)$-modules. 
The similarly defined category of strongly equivariant $\cD(G)$-modules may be identified with the categorical trace of the monoidal category of Harish-Chandra bimodules $\HC(G)$, that is, $\cU(\fg)$-bimodules with an action of $G$ integrating the diagonal $\fg$-action (see \cite{BZGN} for a proof in the derived setting). Modules for $\HC(G)$ are sometimes known as categorical $G$-representations (more formally, these are the weak invariants of a categorical $G$-representation, namely a module category for the categorical convolution algebra $(\cD(G), \ast)$; see \cite{BZGO} for further details). This leads to an interpretation of strongly equivariant $\cD(G)$-modules as characters of categorical representations.

The formalism of factorization homology gives us a satisfying parallel to this classical story in the quantum setting.  Whereas the category  of strongly equivariant $\Dq$-modules (``quantum character sheaves") is attached by factorization homology to the two-torus, the same theory attaches to the circle a monoidal category $\HCq(G)$ of \emph{quantum Harish-Chandra} bimodules (see Section \cite{Safronov-qmm-paper} for discussion).

The fact that characters of quantum Harish-Chandra bimodules are indeed strongly equivariant $\Dq$-modules is therefore simply a formal consequence of the fact that the two-torus is obtained from the circle by taking the Cartesian product with another circle, and that in topological field theory crossing with a circle produces the appropriate form of Hochschild homology, i.e. the universal receptacle for characters.  For example, one may interpret the module $\HKchi$ as coming from the solid torus $S^1\times D^2$, with a line defect along $S^1\times \{0\}$ labeling the conjugacy class.  This is precisely the character $\Dq$-module of the object of $\HCq$ which is $\Oq \otimes_{\Oq^G}\chi$, the quantization of the conjugacy class determined by $\chi$.

This approach to quantum character theory strongly echoes the ideas of the classical character field theory papers \cite{BZN} by Ben-Zvi and Nadler, and subsequently 
in their work \cite{BZGN} with the first author.

\subsection{Quantum Springer theory}\label{sec:quantum springer}
Recall that the main objects of study in this paper, namely the category $\Dqstr{G}$ and $\HKuniv$, may be defined for any connected reductive group $G$ (together with the additional data required to defined the corresponding ribbon category of representations of the quantum group). Although our techniques in this paper are rooted in the type $A$-specific philosophy of Schur-Weyl duality, combining our results with the analogous results in the $\cD$-module setting leads us to the following natural conjecture.

\begin{conjecture}
    The isomorphism (\ref{eqn:LS}) and the results of Theorems \ref{mainthm:Springer} and \ref{mainthm:Springerchi}
     all hold for an arbitrary connected reductive group $G$ (replacing the symmetric group $\SN$ with the Weyl group $W$ of $G$). 
\end{conjecture}

Proving these conjectures would require a deeper understanding of $\q$-deformed Springer theory, as outside of type $A$ we lack  elliptic Schur-Weyl duality, hence the reformulation via intertwiners for the double affine Hecke algebra, which are our main tools in proving Theorem \ref{mainthm:Springer}. 

We believe that the correct approach to these conjectures is to develop a ``quantum Springer theory", which is to say a theory of parabolic induction and restriction functors for categories of strongly equivariant $\Dq$-modules.  The required functors should arise in the factorization homology framework by considering the domain wall given by $\Repq B$, placed at the surface $T^2\times \{\frac12\}$ inside the 3-manifold $T^2\times I$. We plan to return to this construction in future work. 

Given such a formalism, one can the define the notion of a \emph{cuspidal} $\Dq$-module: one killed by parabolic restriction for every proper parabolic subgroup. The generalization of Theorems \ref{thm:HK generates (GJV)} and \ref{thm:DqstrSL2} for other groups $G$ requires the following more general notion: a \emph{$\q$-cuspidal datum} is defined to be a $G$-conjugacy class of pairs $(L,C)$, where $L$ is an elliptic-pseudo Levi subgroup of $G$ (in the sense of \cite{FratilaGunninghamLi}) and $C$ is a simple unipotent cuspidal $\DqG{L}$-module. For example, if $L=C$, a $\q$-cuspidal datum is the same thing as a unipotent simple cuspidal $\Dq$-module. At the other extreme, for any $G$, there is a canonical cuspidal datum $Spr=(H,\OqG{H})$, where $H$ is the maximal torus of $G$, and $\OqG{H}$ is the canonical module. 
The following statement is the $\q$-analogue of Theorem A of \cite{G}. 

\begin{conjecture}\label{conj:q-generalized Springer}
    For each connected reductive group $G$, there is a finite block decomposition of the category $\Dqstr{G}$ indexed by the set of $\q$-cuspidal data. Moreover, the block corresponding to a cuspidal datum $(L,C)$ is equivalent to the category of modules for the smash product of $\cD_\q(Z(L)^\circ)$ by a twisted group algebra of a certain finite group. 
\end{conjecture}
In particular, there is always at least one block in the above decomposition, namely the \emph{Springer block}, corresponding to the Springer cuspidal datum. We additionally conjecture that this block is generated by the universal Hotta-Kashiwara module $\HKuniv$. According to Theorem \ref{thm:HK generates (GJV)}, this is the only block in the case $G=GL_N$. On the other hand, by
Theorem \ref{thm:DqstrSL2}, there are five blocks: the Springer block, and four blocks corresponding to the four distinct simple cuspidal modules. More generally, we expect the cuspidal $\DqG{SL_N}$-modules to be those whose central character $\overline{n} \in \Z/N\Z = Z(SL_N)^\vee$ satisfies $\gcd(n,N) = 1$. Applying the Schur-Weyl functor $F_n$ to such modules produces finite-dimensional modules for the corresponding DAHA in parallel with the results of \cite{CEE} in the rational case.

Finally, we note that the category of strongly equivariant $\Dq$-modules is the category appearing on the $B$-side in the quantum Betti geometric Langlands conjectures for a 2-torus \cite{BZN18}. One may also wonder how this category is related to the corresponding category in the more traditional de Rham quantum geometric Langlands program. In forthcoming work of the first author, an analogous generalized Springer decomposition is proved for generically twisted $\cD$-modules on the moduli stack $\operatorname{Bun}_G(E)$ of $G$-bundles on an elliptic curve $E$ (see \cite{BZN15} and \cite{FratilaGunninghamLi} for some context). In particular, there is an analogous set $\cuspell{G}$ of \emph{elliptic cuspidal data} indexing the blocks. Although the categorical structure of the individual blocks is quite different in the elliptic and the $\q$-cases, one still expects the following comparison.

\begin{conjecture}
    There is a natural bijection between the set of simple cuspidal $\Dq$-modules and simple cuspidal generically twisted $\cD$-modules on $\operatorname{Bun}_G^{ss,0}(E)$.  
\end{conjecture}

Note that the cuspidal $\cD(\fg)$-modules have been explicitly determined for each simply connected quasi-simple group $G$ (see the table on the second page of \cite{Lusztig1984}). As part of ongoing work of the first author with Penghui Li and Dragos Fratila, we will describe how the set of $\q$-cuspidal data (or equivalently, elliptic cuspidal data) for any given group $G$  can be explicitly determined from Lusztig's table, giving a complete description of the category of strongly equivariant $\Dq$-modules. 

\subsection{Skeins dictionary}

Theorem \ref{mainthm:HKuniv} and its corollaries are of independent interest in the study of skein algebras.  The works \cite{BZBJ1,BZBJ2, Cooke,GJS} establish, for any group $G$, a canonical equivalence of categories,
\[
\Dqstr{G} \simeq \SkCat(T^2)\modu,
\]
between the category of strongly equivariant $\Dq$-modules and the free co-completion $\SkCat_G(T^2)\modu$ of the $G$-skein category of the two-torus.

Under this equivalence, $\HKuniv$ maps to the empty skein object: indeed, the presentation we gave for $\HKuniv$ is precisely the presentation of the distinguished object in factorization homology via quantum Hamiltonian reduction given in \cite{BZBJ2}, which is shown in \cite{Cooke} to coincide with the empty skein object.  We immediately obtain an isomorphism between the \emph{skein algebra} $\SkAlg(T^2)$ of the torus $T^2$ -- i.e. the endomorphism algebra in the skein category of the empty object -- and the endomorphism algebra $\End(\HKuniv)$ within strongly equivariant $\Dq$-modules.  To summarize:
\begin{corollary}\label{cor:skeinalg}
We have isomorphisms,
\[
\SkAlg_G(T^2)\cong \End(\HKuniv) \cong \ee\cdot \DAHA{N}{N}^{op}\cdot \ee \cong  (\Coinvq)^{op},
\]
\end{corollary}

Recall that the case $G=\SL_2$ of this isomorphism is a fundamental computation in skein theory due to Frohman and Gelca \cite{Frohman-Gelca}.  Hence, we may regard Corollary \ref{cor:skeinalg} as a generalization of Frohman and Gelca's result from $G=\SL_2$ to the groups $G=\SLN, \GLN$.  As justification for the categorical approach taken in this paper, we note that a completely elementary/computational proof of this corollary appears intractable due to the well-known combinatorial complexity of a generators-and-relations presentation of $\Coinvq$. In the sequel paper \cite{GJVY}, we extend the isomorphism to all of $\DN$, and apply this isomorphism to compute dimensions of $\GLN$- and $\SLN$-skein modules of the three-torus $T^3$.

In this formulation, Corollary \ref{cor:skeinalg} 
may also be compared to a theorem of Morton and Samuelson \cite{Morton-Samuelson}, where analogous descriptions are obtained relating  the Homflypt skein algebra to the $q=t$ specialization of the elliptic Hall algebra.

It should be possible to recover Morton and Samuelson's result from ours, by simultaneously regarding the Homflypt skein as a ``limit" (polynomial interpolation) of the categories $\Repq(\GLN)$, and likewise the elliptic Hall algebra as an inverse limit of spherical double affine Hecke algebras.  It does not seem possible to reverse this process to obtain our result as a consequence of theirs, and indeed understanding the relation at finite rank rather than in the limit was a primary motivation for the present paper.

In the forthcoming paper \cite{GJVY}, we prove a stronger result (using the Morita equivalence discussed in Theorem \ref{thm:HK generates (GJV)}), which identifies certain relative skein algebras with (full rather than anti-spherical) double affine Hecke algebras.  That result bears an analogous relation to recent work \cite{Mellit-and-co}, where the double affine Hecke algebra was realised as a relative skein algebra for the Homflypt skein relations.

\subsection{Structure of the paper} \label{sec:plan} 
The paper is laid out as follows.  In Section \ref{sec-quantum} we give the basic definitions relating to quantum groups and $\Dq$-modules.  Many of these are recollections, but some are new definitions.  In Section \ref{sec:AHA} we recall the basic definitions for double affine Hecke algebras attached to $\GL$ and $\SL$.  In Section \ref{sec:EndHK} we prove Theorem \ref{mainthm:HKuniv}, by showing that a quantum spherical Schur-Weyl duality homomorphism constructed there is an isomorphism.  In Section \ref{sec:End HK via shift}, we apply the well-known shift isomorphism to compute endomorphism algebras of quantum Hotta-Kashiwara modules $\HKe$ and $\HKchi$, proving Theorems \ref{mainthm:Springer}, \ref{mainthm:decomp}, and \ref{mainthm:Springerchi}.  The paper ends with Section \ref{sec:End via intertwiners}, which can be read independent of the other results in the paper.  It proves via the method of intertwiners a statement purely about certain representations of double affine Hecke algebras, which taken together with Theorem \ref{thm:HK generates (GJV)} provides an alternative proof of Theorems \ref{mainthm:Springer}, \ref{mainthm:decomp}, and \ref{mainthm:Springerchi}.  We include it in the present paper rather than the forthcoming paper as it is closer in spirit to the results obtained here, despite the fact that its application depends on the results from \cite{GJVY}.

\begin{remark}
On the eve of posting this paper, we received a beautiful pre-print \cite{Wen} from Joshua Wen, which establishes a quantum Harish-Chandra isomorphism similar to our Theorem \ref{mainthm:HKuniv} for quantum multiplicative quiver varieties using the radial parts construction.
\end{remark}

\subsection{Acknowledgments and dedication}
We have benefited during this work from insightful conversations with David Ben-Zvi, Pavel Safronov, Peter Samuelson, and Jos\'e Simental,.  The main idea of this paper was suggested to us by Tom Nevins, who outlined to the second author in 2012 the idea for a quantum Springer theory approached via quantum-Hotta Kashiwara modules.  Tom was a generous and inspiring mathematician, and a kind soul; we dedicate this paper in honour of his memory.

The first author was partially supported by NSF grant DMS-2202363. The second author was partially supported by ERC Starting Grant no. 637618, and by the Simons Foundation award 888988 as part of the Simons Collaboration on Global Categorical Symmetry.
The third author was partially supported by Simons Foundation Collaboration Grants 319233 and 707426. 
The authors are grateful to the International Centre for Mathematical Sciences Research in Groups Programme, and to Aspen Center for Physics, which is supported by National Science Foundation grant PHY-1607611, who hosted research visits during which parts of this work was undertaken.

\section{Quantum groups, their differential operators and Hotta-Kashiwara modules}
\label{sec-quantum}
In this section we recall basic definitions about quantum groups, braided quantum coordinate algebras, and quantum differential operators.  Finally, we recall the notion of strongly equivariant $\Dq$-modules, and we give the construction of the Hotta-Kashiwara modules $\HKchi$.

\begin{notation}\label{not:groundfieldquantum}
Fix a positive integer $N$. In this section, we let $\cR$ denote a $\Q$-algebra that is an integral domain together with an element $\q^{\frac 1N} \in \cR^\times$ such that $\frac{\q^m-1}{\q-1}$ is invertible for all $m\geq 1$, and let $\cK$ denote its field of fractions. Our main example for $\cR$ is the local ring $\cR = \Q[\q^{\frac 1N}]_{(\q-1)}$ at $\q=1$, in which case $\cK = \Q(\q^{\frac 1N})$.  Unless otherwise specified, everything in this section will be defined over the base ring $\cR$.
\end{notation}

\begin{remark} 
 \label{rem: various q}
 To straddle the various conventions in the literature, we have had to introduce  different fonts for our parameters:  (teletype) $\q$ for quantum groups of Section \ref{sec-quantum}, (Roman) $q$ (respectively  bold Roman $\Sq$) for the loop parameter of the DAHA of type $\GL$ in Section \ref{sec: GL DAHA} (resp. type $\SL$ in Section \ref{sec: SL DAHA}). As we shall see below, the quadratic parameter in the Hecke algebras is usually denoted $t$, and for Schur-Weyl considerations we later specialize $t=\q$.  All of the parameters come together in Section \ref{sec:EndHK}.
\end{remark}

\begin{notation}
We will have occasion to use the quantum integers denoted $[k]_r$ for $k \in \Z$ and for various elements $r\in\cR$.  We take the (unbalanced) convention
$[k]_{r} = \frac{r^k - 1}{r - 1}$
and
$[n]_{r}! = [n]_{r} \cdots [2]_{r} [1]_{r}$. 
As usual ${\genfrac{[}{]}{0pt}{}{n}{k}}_r = \frac{[n]_r [n-1]_r \cdots [n-k+1]_r}{[k]_r \; [k-1]_r \; \, \cdots  \, \; [1]_r}$.
\end{notation}

\subsection{The quantum groups $\Uqgl, \Uqsl$}
We may work over  Lusztig's integral form of the quantum group over $\cR$, but as $[n]_\q$ is a unit in $\cR$, there is no  appreciable difference working over this integral form as compared to  other conventions. 
We refer to \cite{KlSch} for detailed definitions, in particular the Serre presentation of the quantum groups $\Uqgl$ and $\Uqsl$, the formulas for $R$-matrices, and the Peter-Weyl theorem.

We will consider the $\cR$-linear braided tensor categories $\Rep_\q(\GLN)$ and $\Rep_\q(\SLN)$ of integrable $\Uqgl$-modules (resp. $\Uqsl$-modules). Note that our hypotheses on the ring $\cR$ mean that the Weyl modules $V(\lambda)$ form a collection of compact projective generators as $\lambda$ ranges over the set of integral dominant weights. There are no nontrivial homomorphisms nor extensions between different Weyl modules. Thus every object of $\Rep_\q(G)$ may be written as a direct sum 
\begin{equation}\label{eqn:Repdecomp}
V \cong \bigoplus_\lambda V(\lambda) \otimes_\cR M_\lambda
\end{equation}
where $\M_\lambda$ is an $\cR$-module.
The braiding is specified by an $R$-matrix over the ring $\cR$.

\subsection{The vector and fundamental representations} The vector representation $\cR^N$ for either $\Uqgl$ and $\Uqsl$ will be simply denoted $V$.  We fix $v_{1},\ldots, v_N$ to be the standard basis for $V$.

Let $\tau:V\ot V$ denote the tensor flip, $\tau(v\ot w)=w\ot v$.
 The braiding, $\sigma_{V,V}=\tau\circ R$,
 where $R = R_{V,V}$ denotes the $R$-matrix for $\Uqgl$,
 satisfies a Hecke relation
$$(\sigma_{V,V}-\q)(\sigma_{V,V}+\q^{-1})=0.$$

The $\SLN$ braiding on the vector representation is equal to the $\GLN$ braiding, multiplied by a factor of $\q^{-\frac{1}{N}}$.  To avoid confusion, we will reserve the notation $R$ and $\sigma_{V,V}$ for the $\GLN$ $R$-matrix and braiding, and write $\q^{-\frac{1}{N}}R$ and $\q^{-\frac{1}{N}}\sigma_{V,V}$ to reference the $\SLN$ $R$-matrix and braiding. 

The operators $T_i=\id_{V^{\otimes (i-1)}}\otimes \sigma_{V,V}\otimes \id_{V^{\otimes n-i-1}}$ (in both the $\GLN$ and $\SLN$ case) define an action of the finite Hecke algebra $\Hf_n$ on $V^{\ot n}$ with quadratic parameter $\q$.
We  denote by $\wedgeN{k}: = \bigwedge\nolimits_\q^k(V)$ the $k$th fundamental representation of $U_\q(\glN)$. In particular $\wedgeN{N}$ is the one-dimensional determinant representation. 
We can embed $\wedgeN{N}$  into $V^{\otimes N}$ as the span of 
\begin{gather} \label{eq:z} 
z = \eN \cdot v_1 \otimes v_2 \otimes \cdots \otimes v_N
= \frac{1}{[N]_{\q^{-2}}!}  \sum_{\sigma \in \SN} (-\q^{-1})^{\ellen(\sigma)}
v_{\sigma(1)} \otimes  v_{\sigma(2)} \otimes \cdots \otimes v_{\sigma(N)}.
\end{gather}

In the case of $\slN$ the  vector $z$ defined by
\eqref{eq:z}  is now an  invariant vector. 
Note as in \eqref{eq:triv sign} below for $\q=t$, that $\eN$ is the sign idempotent in $\Hf_N$ .

\subsection{The quantum coordinate algebra}

\begin{definition}\label{defn:REA} The reflection equation algebra of type $\GLN$, denoted $\cO_\q(\MatN)$, is the $\cR$-algebra generated by symbols $\ellgen^i_j$, for $i,j = 1,\ldots N$, subject to the relations,
 $$R_{21}L_1R_{12}L_2 =  L_2R_{21}L_1R_{12},$$
where $L := \sum_{i,j} \ellgen^i_j E^j_i$ is a matrix with entries the generators $\ellgen^i_j$, and for a matrix $X$, we write $X_1=X\ot \id_V$, $X_2=\id_V\ot X$, so that the matrix equation above is equivalent to the list of relations, for $i,j,n,r\in\{1,\ldots,N\}$:
\begin{equation}\sum_{k,l,m,p}R^{ij}_{kl}\ellgen^l_mR^{mk}_{np}\ellgen^p_r = \sum_{s,t,u,v} \ellgen^i_sR^{sj}_{tu}\ellgen^u_vR^{vs}_{nr}.
\label{eqn:OqRelns}\end{equation}
\end{definition}
\begin{remark} We note that, since $R$ appears quadratically on both sides of the defining relation, the relations are unchanged by replacing $R \leadsto \q^{-\frac{1}{N}}R$.
\end{remark}

\begin{proposition}[\cite{JW}]
The element,
$$\detq(L) := \sum_{\sigma\in \SN} (-\q)^{\ellen(\sigma)}\cdot \q^{e(\sigma)} \ellgen^1_{\sigma(1)}\cdots \ellgen^N_{\sigma(N)},$$
is central in $\cO_\q(\MatN)$.
\end{proposition}
Here $\ellen(\sigma)$ denotes the length, i.e. the number of pairs $i<j$ such that $\sigma(i)>\sigma(j)$, and $e(\sigma)$ denotes the excedence, i.e. the number of elements $i$ such that $\sigma(i)>i$.

\begin{definition}
The quantum coordinate algebras $\cO_\q(\GLN)$ and $\cO_\q(\SLN)$ are the algebras obtained from $\cO_\q(\MatN)$, by inverting, 
respectively specializing to one, the central element $\detq(L)$.  That is,
$$\cO_\q(\GLN) = \cO_\q(\MatN)[\detq(L)^{-1}],\qquad \cO_\q(\SLN) = \cO_\q(\MatN)/\langle \detq(L)-1 \rangle.$$
\end{definition}

The $\cR$-algebra $\Oq$ is naturally an algebra object in the tensor category $\Repq(G)$, in other words it is a locally finite $\Uq$-module algebra.  We recall the quantum Harish-Chandra isomorphism, 
\[
\Oq^G \cong \cR[H]^W,
\]
 which holds for generic $\q$, and identifies the algebra $\Oq^G$ of $\Uq$-invariant elements of $\Oq$ with the $W$-invariant functions on the maximal torus $H$.
 
 \begin{notation}\label{not:ck-def}
 We denote by $\cas{1},\ldots,\cas{N}$ the canonical generators -- ``quantum Casimirs" -- which are obtained as the quantum trace of the $k$th fundamental representation $\wedgeN{k}$ (see \eqref{eqn:c_k} below for a depiction).  In particular, we have $\cas{N}=\detq(L)$.  
 Working over $\K$ the elements $\cas{k}$ generate the center of $\Oq$.
 \end{notation}

\subsection{The quantum Harish--Chandra category $\HCq(G)$}
\label{sec:quantum HC}
The discussion in this subsection is not strictly necessary to establish the results in this paper, but is meant to aid the reader in interpreting those results.

Let $A$ be some algebra object in the category $\Repq(G)$.  Recall that this means $A$ is an object of $\Repq$, with an associative product $m: A\otimes A \to A$ internal to $\Repq(G)$, i.e. $m$ is a homomorphism of $U_\q(\mathfrak{g})$-modules.

\begin{definition}
The category of \defterm{weakly equivariant} (left) $A$-modules consists of objects $M\in \Repq(G)$, equipped with associative action maps $\Dq\ot M\to M$ internal to $\Repq(G)$.
\end{definition}

The quantum Harish--Chandra category $\HCq(G)$ is the category of $\Oq$-modules internal to the category $\Repq G$.  It carries the cp-rigid monoidal structure of relative tensor product. See \cite{BZBJ2, Safronov-qmm-paper} for further details.

\begin{remark}
Specializing the category $\HCq(G)$ at $\q=1$ (that is, working over the ring $\kappa = \cR/(\q-1)$), we obtain the category $\mathrm{QCoh}([G/G])$ of quasi-coherent sheaves on the quotient stack $[G/G]$. On the other hand, as explained in \cite{Safronov-qmm-paper}, the quasi-classical degeneration $\q\to 1$ of $\HCq(G)$ becomes the classical category of Harish-Chandra $U(\mathfrak{g}$)-bimodules.
\end{remark}

\subsection{Quantum differential operators}
The algebra of quantum differential operators on $G$, which we denote by $\cD_\q(G)$, was studied in many different settings.  The presentation below as a twisted tensor product is adapted from the paper \cite{VV} (see also \cite{BrochierJordan2014}), and hence matches the conventions of \cite{Jordan2008} (see footnote 3 there, however) and \cite{Jordan-Vazirani}.

\begin{definition} \label{def:Dq} For $G=\GLN$, or $\SLN$, the algebra $\cD_\q(G)$ is the twisted tensor product,
\begin{gather}
\label{eq-Dq=Oq-Oq}
\cD_\q(G) = \cO_\q(G) \mathbin{\widetilde{\ot}} \cO_\q(G).
\end{gather}
Denoting $a^i_j$ and $b^i_j$ to be the generators of the first and second factors (henceforth denoted $\cA$ and $\cB$), the cross relations are given in matrix form by:
\begin{align}B_2R_{21}A_1 &= R_{21}A_1R_{12}B_2R_{21}, & \textrm{if $G=\GLN$} \label{eqn:DqRelns}\\
B_2R_{21}A_1 &= R_{21}A_1R_{12}B_2R_{21}\q^{-2/N}, & \textrm{if $G=\SLN$}\label{eqn:DqRelns2}
\end{align}
where $A= \sum_{i,j}a^i_jE^j_i$ and $B=\sum_{i,j}b^i_j E^j_i$.
\end{definition}

\begin{remark}\label{rem:not just V}
The matrices $A$ and $B$ here are defined using the defining representation $V$ of $U_\q(\g)$, and its dual $V^*$.  We note for later reference that in fact for any representation $X$  we can define a matrix $A_X\in \Oq \otimes \End_{\cR}(X)$, which will satisfy the relations of \eqref{eqn:OqRelns}, with $R$ replaced by $R_{X,X}$.  Similarly, in the algebra $\cD_\q(G)$ we can define matrices $A_X$ and $B_Y$ for any representations $X$ and $Y$, and their cross relations will be as in \eqref{eqn:DqRelns}, with $R$ replaced by $R_{X,Y}$.  We refer to \cite{Jordan2008} for details about these matrices and the resulting diagrammatic calculus for computing with $\cD_\q(G)$.
\end{remark}

\begin{notation}\label{def:iA and iB}
  We denote by $i_{\cA}$ and $i_{\cB}$  the inclusions into the first and second tensor factor of \eqref{eq-Dq=Oq-Oq}.
  We denote by $\iBt$ the inclusion of $\Oq^G$ into $\Oq$ followed by $i_\cB$, whose image is denoted $\cB^G$.
\end{notation}

Observe that $i_{\cA}\ot i_{\cB}:\Oq\ot\Oq \to\cD_\q(G)$ is the tautological isomorphism of $U_\q(\g)$-modules (however it is not an algebra homomorphism).

\begin{definition}\label{def:qmm}
The quantum moment map,
\[
\ad: \Oq \to \Dq
\]
is defined by the matrix formulas by $\ad(L)=B^{-1}ABA^{-1}$.
\end{definition}

We have the Rosso homomorphism $\Rosso:\Oq\to U_\q(\g)$.  Let $\rhd$ denote the quantum adjoint action of $U_\q(\g)$ on $\Dq$.  Then the homomorphism $\ad$ satisfies the quantum moment map equation,
\begin{equation}
\ad(\ell)y = (\Rosso(\ell)_{(1)}\rhd y) \ad(\Rosso(\ell)_{(2)}), \qquad \textrm{for $\ell\in\Oq$, $y\in\Dq$.}\label{eqn:qmm}
\end{equation}
We refer to \cite{Safronov-qmm-paper} for a more in-depth discussion of quantum moment maps.
\begin{remark}
    If we work over $\cK$, the Rosso homomoprhism $\Rosso$ is an algebra embedding, and pulling back along $\Rosso$ defines a fully faithful functor $\Repq(G) \to \Oq\modu$.
\end{remark}

\begin{notation}\label{not:C}
We require notation for the following $\cR$-submodules  
of $\Dq$:
\begin{align*}
C(\ad) &:=  \{\ad(\ell)- \epsilon(\ell) \,\,|\,\, \ell\in\Oq\},\\
C({i_\cA}) &:= \{ i_{\cA}(\ell)- \widetilde{\epsilon} (\ell) \,\,|\,\, \ell\in\Oq\},\\
\Cchi &:= \{ i_{\cB}(y)- \chi(y) \,\,|\,\, y\in\Oq^G\}
\end{align*}
where $\chi$ is an algebra homomorphism from $\Oq^G$ to $\cR$, i.e., a character.

Here, we denote by $\widetilde{\epsilon}$ the linear morphism $\Oq\to\cR$ determined on matrix coefficients $W^*\otimes W$ 
for a Weyl module $W$
by the formula,
\[
W^*\otimes W \xrightarrow{D\otimes 1} W^* \otimes W \xrightarrow{\epsilon} \cR,
\]
where $D$ denotes the natural transformation of $\id_{\Repq(G)}$ depicted on the left below. Hence $\tilde\epsilon$ is denoted on the right below:

\begin{center}
\begin{tikzpicture}

\begin{scope}[shift={(0,0)}]
    \node at (-1,2) {$D =$};
\begin{knot}[
consider self intersections,
clip width=5
]
\strand [ultra thick]
(0,2)
to [out=down, in=right]
(-0.5,.75)
to [out=left, in=down]
(-0.75,1)
to [out=up, in=left]
(-0.5,1.25)
to [out=right, in=up]
(0,0) ;

\strand [ultra thick] (0,2)
to [out=up, in=left]
(0.5,3.25)
to [out=right, in=up]
(0.75,3)
to [out=down, in=right]
(0.5,2.75)
to [out=left, in=down]
(0,4);

\flipcrossings{1,2}
\end{knot}
\end{scope};

\begin{scope}[shift={(5,0)}]
    \node at (-2,2) {$\widetilde{\epsilon} :=\epsilon \circ (D \ot 1)=$};
\begin{scope}[scale=0.8]
\begin{knot}[
consider self intersections,
clip width=5
]
\strand [ultra thick]
(0,2)
to [out=down, in=right]
(-0.5,.75)
to [out=left, in=down]
(-0.75,1)
to [out=up, in=left]
(-0.5,1.25)
to [out=right, in=up]
(0,0) ;

\strand [ultra thick] (0,2)
to [out=up, in=left]
(0.5,3.25)
to [out=right, in=up]
(0.75,3)
to [out=down, in=right]
(0.5,2.75)
to [out=left, in=down]
(0,4);

\strand [ultra thick] (0,4)
to [out=up, in=up]
(2,4)
to [out=down, in=up]
(2,0);

\flipcrossings{1,2}
\end{knot}
\end{scope};
\end{scope};

\end{tikzpicture}
\end{center}

\end{notation}

\begin{proposition}\label{prop:containment}
We have a containment of right ideals,
\[
C(\ad)\cdot \Dq \subset C(i_{\cA})\cdot \Dq.
\]
Equivalently the surjective homomorphism of right $\Dq$-modules, \[\Dq\to (C(i_{\cA})\cdot \Dq)\big \backslash \Dq,\] factors canonically to a surjective homomorphism,
\[
(C(\ad)\cdot \Dq)\big \backslash \Dq\to (C(i_{\cA})\cdot \Dq)\big \backslash \Dq,
\]
making the following diagram commute:
 \begin{equation}
 \begin{tikzcd}
 \Dq \arrow[d] \arrow[r] & C(i_\cA)\cdot\Dq \, \backslash 
\Dq \\
C(\ad)\cdot \Dq \, \backslash \Dq \arrow[ur, dashed, twoheadrightarrow]
\end{tikzcd}.
 \end{equation}
\end{proposition}
\begin{proof}
It will be convenient to use the matrix notation from Definition \ref{def:qmm}.  Since the matrices $A$ and $B$ are invertible (with inverse given by the antipode), we may rewrite $C(\ad)\cdot \Dq = C'(\ad)\cdot\Dq$,
where 
\[
C'(\ad) =  \{\textrm{matrix coefficients of the matrix } B^{-1}A- AB^{-1}\}.\]  It will suffice to show that $C'(\ad)$ is already zero in $(C(i_\cA)\cdot \Dq)\big \backslash \Dq$.  Clearly the matrix $AB^{-1}$ reduces to $B^{-1}$ modulo $C(i_\cA)\cdot \Dq$.  In order to evaluate $B^{-1}A$, we must first PBW order it so that the $A$'s are to the left, and then apply the coinvariant relation to the $\cA$ factor.  We compute this in Figure \ref{fig:Binverse A}, showing that $B^{-1}A$ also reduces modulo $C(i_\cA)\cdot \Dq$ to $B^{-1}$, hence finally that $C'(\ad)$ is zero in $(C(i_\cA)\cdot \Dq)\big \backslash \Dq$, as required.
\end{proof}

\begin{figure}[h]
    \centering

\begin{tikzpicture}
\begin{scope}[shift={(-0.6,-0.5)} ]
    \node[scale=1.2] at (-0.2,2) { $ B^{-1} A=$};
\end{scope}

\begin{scope}[shift={(0,0)}, scale=0.5]
\draw[blue,knot,ultra thick]  
(4,0)
to [out=up, in=down]
(4,1)
to [out=up, in=down]
(3,2) 
to [out=up, in=down]
(3,4) 
to [out=up, in=down]
(8,8) ;

\draw[dotted, thick] (0.75, 2.5)--(8, 2.5); 
\draw[dotted, thick] (0.75, 4.5)--(8, 4.5);

\draw[black,knot,ultra thick]
(5,0)
to [out=up, in=down]
(5,1) 
to [out=up, in=down]
(7,2)
to [out=up, in=down]
(7,4)
to [out=up, in=down]
(4,8);

\draw[black,knot,ultra thick]
(2.5,3)
to [out=up, in=right]
(1.75,4) 
to [out=left, in=down]
(1,5)
to [out=up, in=down]
(6,8);

\draw[black,knot,ultra thick]
(1,6.5)
to [out=down, in=up]
(1,6)
to [out=down, in=left]
(3,5.5) 
to [out=right, in=up]
(5.5,3)
to [out=down, in=right]
(4.5,2)
to [out=left, in=down]
(3.5,2.75)
to [out=up, in=right]
(2.75,3.5)
to [out=left, in=up] 
(2,3)
to [out=down, in=down]
(2.5,3);

\end{scope}

\begin{scope}[shift={(5.5,0)}, scale=0.5]
\draw[blue,knot,ultra thick]  
(4,0)
to [out=up, in=down]
(4,1)
to [out=up, in=down]
(3,2) 
to [out=up, in=down]
(3,4) 
to [out=up, in=down]
(8,8) ;

\draw[red,knot,ultra thick]
(0.5,6.75) to [out=right, in=left] (2.5,8.75);
\draw[black,knot,ultra thick]
(5,0)
to [out=up, in=down]
(5,1) 
to [out=up, in=down]
(7,2)
to [out=up, in=down]
(7,4)
to [out=up, in=down]
(4,8); 
\draw[red,knot,ultra thick] 
(4,8) 
to [out=up, in=up]
(2,9)
to [out=down, in=left]
(2.5,8.25)
to [out=right, in=down]
(2.75,8.5) 
to [out=up, in=right]
(2.5,8.75);

\draw[black,knot,ultra thick]
(2.5,3)
to [out=up, in=right]
(1.75,4) 
to [out=left, in=down]
(1,5)
to [out=up, in=down]
(6,8);

\draw[red,knot,ultra thick] 
(0.5,6.75)
to [out=left, in=down]
(0.25,7) 
to [out=up, in=left]
(0.5,7.25)
to [out=right, in=up]
(1,6.5);
\draw[black,knot,ultra thick]
(1,6.5)
to [out=down, in=up]
(1,6)
to [out=down, in=left]
(3,5.5) 
to [out=right, in=up]
(5.5,3)
to [out=down, in=right]
(4.5,2)
to [out=left, in=down]
(3.5,2.75)
to [out=up, in=right]
(2.75,3.5)
to [out=left, in=up] 
(2,3)
to [out=down, in=down]
(2.5,3);

\end{scope}

\begin{scope}[shift={(9,-0.5)} ]
    \node[scale=1.2] at (1.5,2) {$ =$};
    \node[scale=1.2] at (4,2) {$ = B^{-1}$};
\end{scope}

\begin{scope}[shift={(9.5,0)}, scale=0.5]  
\draw[blue,knot,ultra thick]
(3.5,0)
to [out=up, in=down]
(6,8) ;

\draw[black,knot,ultra thick]
(4,8)
to [out=down, in=up]
(3,5.5)
to [out=down, in=right]
(2,4.15)
to [out=left, in=down]
(1.75,4.5)
to [out=up, in=left]
(2,4.85);

\draw[black,knot,ultra thick]
(2,4.85)
to [out=right, in=up]
(5,0) ;

\end{scope}

\end{tikzpicture}

\caption{We compute $B^{-1}A$ modulo $C(i_\cA)\Dq$.  The first diagram, reading from bottom to top, is a composition of the coproduct in $\Oq$, the map $B\mapsto B^{-1}$ on the leftmost resulting factor ($\cB)$ of $\Oq$ in $\cB\otimes \cA = \Oq\otimes \Oq$, and finally the multiplication in $\Dq$ is computed at the top.  The second diagram is equivalent to the first one modulo $C(i_\cA)$, as we have applied ${\color{red}{\widetilde\epsilon}}$ to the $\cA$-factor.  Finally in the third diagram we have simplified the diagram by applying isotopies.  To make the isotopies easier to follow, we've colored an inert backmost strand blue.}
    \label{fig:Binverse A}
\end{figure}


\begin{notation}
We let $\Dq^G$ (resp. $\cA^G$, $\cB^G$)  denote the invariant subalgebra of $\Dq$ (resp. $\cA^G$, $\cB^G$) with respect to the quantum adjoint action $\rhd$. 
\end{notation}

\begin{definition} \label{def: strongly equivariant}
A weakly equivariant left (respectively, right) $\Dq$-module $M$ is \defterm{strongly equivariant} if the equation 
\[
\ad(\ell)\cdot m = \Rosso(\ell)\cdot m,\qquad \textrm{respectively} \qquad
m\cdot\ad(\ell) = S(\Rosso(\ell)).m\]
holds for all $\ell\in \Oq$ and $m\in M$.  The category of strongly equivariant $\Dq$-modules is denoted $\Dqstr{G}$.
\end{definition}

If we work over a base field $\K$ (so that we treat the quantum parameter $\q$ as generic), then there we have the following more elementary formulation of strongly equivariant $\Dq$-modules.
\begin{definition}
Let $A$ be an $\K$-algebra and $M$ an $A$-module.  We say $M$ is \defterm{locally finite} if, for all $m\in M$, we have $\dim_\K(A\cdot m)<\infty$.
\end{definition}

\begin{proposition}\label{prop:full subcategory}
Suppose that the base ring is the field $\K$.  Then the category of strongly equivariant left (respectively, right) $\Dq$-modules is equivalent to the full subcategory of left (respectively, right) $\Dq$-modules whose restriction to $\Oq$ via $\ad$ is locally finite.
\end{proposition}

\begin{proposition}\label{prop:coinv}
Suppose that the base ring is $\K$, that $M$ (respectively $N$) is a strongly equivariant left (resp. right) $\Dq$-module.  Then we have isomorphisms of vector spaces, 
$$
\rightquotient{C(\ad)\cdot M}{M} \, \cong M_G \cong M^{G} \quad \text{and} \quad
\leftquotient{N}{N \cdot \, C(\ad) } \, \cong N_G \cong N^{G},
$$
where $M_G$ and $M^G$ (resp. $N_G$ and $N^G$) denote
 the coinvariants and invariants.
\end{proposition}

\subsection{Quantum Hotta-Kashiwara modules}
\begin{notation}
Given a character $\beta$  of an algebra, we write $\r(\beta)$ for the corresponding 1-dimensional module.

Given algebras $R$ and $S$, an algebra homomorphism $\phi:R\to S$, a right $S$-module $N$, and a left $R$-module $M$, we will denote $N\rt{\phi}M$ as a shorthand to denote the relative tensor product $\phi^*(N)\rt{R}M$.  In the case $N$ is a $S$-$S$ bimodule, $N\rt{\phi}M$ inherits the structure of a left $S$-module.
 \end{notation}
\begin{definition}
Let $\HKuniv$ denote the \defterm{universal quantum Hotta-Kashiwara module},
\[
\HKuniv = \Dq\rt{\ad}\r(\epsilon), \]
where $\epsilon$ denotes the counit on $\Oq$, i.e. $\epsilon(\ellgen^i_j)=\delta_{i,j}$.
\end{definition}

\begin{remark}\label{rmk:ideal} Alternatively, $\HKuniv$ is the quotient of $\Dq$ by the left ideal $J$ generated by elements of
$C(\ad)$. 
Since $\Dq^G$ commutes with $\ad(\ell)-\epsilon(\ell)$ for any $\ell\in \Oq$, it follows that the regular $\Dq$-$\Dq^G$-bimodule action on $\Dq$ descends through the quotient, making $\HKuniv$ a $\Dq$-$\Dq^G$-bimodule.
\end{remark}

\begin{proposition}\label{prop: Hom HK gives invariants}
Let $M$ be a strongly equivariant $\Dq$-module. Then
\[
\Hom_{\Dqstr{G}}(\HKuniv,M)\cong \Hom_{\Repq(G)}(\un,M) =: M^G,
\]
where ${\un}$ denotes the trivial representation.
In particular, $\HKuniv$ is a projective object in $\Dqstr{G}$.
\end{proposition}
\begin{proof}
We have
\[
\Hom_{\Dqstr{G}}(\HKuniv,M)\cong \Hom_{\Oq\modu_{\Repq(G)}}(\mathbf 1,M),
\]
where $\Oq\modu_{\Repq(G)}$ denotes the category of $\Oq$-modules internal to $\Repq(G)$. As $M$ is assumed to be strongly equivariant, this is the just the same as $\Hom_{\Repq}(\un,M)= M^G$ as required. 
\end{proof}

\begin{remark} \label{rem:Dist}
    More generally, we have
    \[
    \Hom_{\Dqstr{G}}(\Dist(W), M) = \Hom_{\Repq(G)}(W,M)
    \]
    where for an object $W\in \Repq$, we denote the strongly equivariant module
\[
\Dist(W) = \Dq \otimes_{\ad} \mathfrak R^\ast(W).
\]
\end{remark}

\begin{definition}
Fix a character $\chi$ of $\Oq^G$.
We define the $\Dq$-module
\[
\HKchi = \HKuniv \rt{\cB^G} \r(\chi),
\]
which we call a \defterm{quantum Hotta-Kashiwara module}.
\end{definition}

\begin{remark}
Unwinding the definitions one sees that the quantum Hotta-Kashiwara module may be presented as
\[
\HKchi = \leftquotient{\Dq }{ ( \Dq\cdot C(\ad) + \Dq\cdot \Cchi)},
\]
as in equation \eqref{eqn:HKchi} from the introduction.
\end{remark}

\begin{remark} We note that $\HKuniv$ and each $\HKchi$ are strongly equivariant as $\Dq$-modules.
\end{remark}

\begin{proposition}\label{prop: HK chi not zero} Suppose that the base ring is $\K$,  the natural inclusion $i_{\cB^G}: \Oq^G\to \Dq$ descends to an isomorphism of vector spaces,
	\[
	\Oq^G\cong \rightquotient{C(i_{\cA})\cdot \HKuniv}{\HKuniv},
	\]
\end{proposition}
\begin{proof}
We compute,
\begin{align*}
\rightquotient{C(i_{\cA})\cdot\HKuniv}{\HKuniv} &= \rightquotient{C(i_{\cA})\cdot\Dq}{\Big(\leftquotient{\Dq}{\Dq\cdot C(\ad)}\Big)} \\
&= \leftquotient{\left(\rightquotient{C(i_{\cA})\cdot\Dq}{\Dq}\right) }{\Dq\cdot C(\ad)} \\
&\cong \Big(C(i_{\cA})\cdot\Dq \,\,\big\backslash\,\, \Dq\Big)^G\\
&\cong \Oq^G
\end{align*}
where the penultimate isomorphism follows from Proposition \ref{prop:coinv}, and the last isomorphism follows from the PBW property for $\Dq$. The isomorphism is compatible with the inclusion $i_\cB:\Oq\to\Dq$, and hence the claim follows.
\end{proof}
\begin{corollary}\label{cor:HKchi}
Suppose that the base ring is $\K$. For any character $\chi$ of $\Oq^G$, the image of $1\in\Dq$ in the quantum Hotta--Kashiwara module $\HKchi$ is nonzero.
\end{corollary}
\begin{proof}
We compute,
\begin{eqnarray*}
\HKchi^G \cong \doublequotient{C(\ad)\cdot\Dq}{\Dq}{( \Dq\cdot C(\ad) + \Dq\cdot \Cchi)}.
\end{eqnarray*}
By Proposition \ref{prop:containment}, this surjects onto
\[
\doublequotient{C(i_\cA)\cdot\Dq}{\Dq}{( \Dq\cdot C(\ad) + \Dq\cdot \Cchi)},
\]
which is one-dimensional by Proposition \ref{prop: HK chi not zero}, and spanned by the image of $1\in\Dq$.
\end{proof}

The following definition is a $\q$-analogue of the corresponding notion for $\cD(G)$-modules defined by Ginzburg \cite{Ginsburg1989}.
\begin{definition} \label{def:admissible}
    A strongly equivariant $\Dq$-module $M$ is called \emph{admissible} if it is $\cB^G$-locally finite. We write $\Dqadm{G}$ for the full subcategory of $\Dqstr{G}$ consisting of admissible modules. 
\end{definition}

In analogy with the $\cD(G)$ and $\cD(\mathfrak{g})$ setting, we will use the term $\q$-character sheaf for an admissible strongly equivariant $\Dq$-module.

\begin{definition}\label{def:central}
    Let $C$ denote an orbit for the action of $W^{\aff}$ on the torus $H$. 
    We say that an admissible $\Dq$-module $M$ has \emph{infinitesimal character $C$} if, 
    for any simultaneous eigenvalue $\chi:\cB^G\to \cK$ for the action of $\cB^G$ on $M$, the character $\chi$ maps to $C$ under the projection 
    $\Spec(\cB^G) \cong H/W \to H/W^{\aff}$. 
    We write $\DqadmC{G}$ for the subcategory of admissible modules with infinitesimal character $C$. 
    We call an admissible module \emph{unipotent} if it has infinitesimal character in the $W^{\aff}$-orbit of $1\in H$. 
\end{definition}

\begin{remark}
Just as in the classical case \cite[Theorem 1.3.2]{Ginsburg1989}, one can show that there is an orthogonal decomposition
    \[
\Dqadm{G} = \bigoplus_{C\in H/W^{aff}} \DqadmC{G}.
    \]
We will not need the result in this paper, so we omit the proof. In the case $G=\GLN$, this fact follows readily by considering the functor $\Upsilon$ from Section \ref{sec:shift}, which is known to be an equivalence after the results of \cite{GJVY}. 
\end{remark}

\section{Double affine Hecke algebras}
\label{sec:AHA}

\begin{notation}\label{not:groundfieldAHA}
    When discussing the finite and affine Hecke algebra we will use the following notation. Fix a $\Q$-algebra $\cR$ which is an integral domain and contains fixed element $t \in \cR^\times$, such that $\frac{t^m-1}{t-1}$ is invertible for all $m \neq 0$. Let $\cK$ denote the field of fractions of $\cR$. Unless stated otherwise everything will be defined with base ring $\cR$.
\end{notation}

\begin{definition} The finite Hecke algebra $\Hf_n(t)$  is the $\cR$-algebra with  generators $T_1,\ldots T_{n-1}$, and relations:
\[
T_iT_{i+1}T_i = T_{i+1}T_iT_{i+1}, \textrm{ for $i=1,\ldots, n-2$},\]
\[\quad T_iT_j=T_jT_i, \textrm{ if $|i-j|\geq 2$ },\qquad
(T_i-t)(T_i+t^{-1})=0, \textrm{ for $i=1,\ldots n-1$}.
\]
\end{definition}

Given a reduced word $w = s_{i_1} \cdots s_{i_m} \in \Sn$ we have a corresponding well-defined element $T_w = T_{i_1} \cdots T_{i_m} \in \Hf_n(t).$  We will write $\ellen(w) =m$ for the length of $w$. By convention $T_{\id} = 1.$ Note that the set
$\{ T_w | w \in \Sn\}$ forms a basis of $\Hf_n(t).$

The finite Hecke algebra has two one-dimensional representations. Corresponding to the trivial and the sign representation, respectively, we have the idempotents
\begin{gather} \label{eq:triv sign}
   \etrivn =  \frac{1}{[n]_{t^2}!} \sum_{w \in \Sn} t^{\ellen(w)} T_w, \qquad     \en = \frac{1}{[n]_{t^{-2}}!} \sum_{w \in \Sn} (-t^{-1})^{\ellen(w)} T_w.
\end{gather}
These satisfy $(T_i -t) \etrivn = 0$ and $(T_i + t^{-1}) \en = 0$, for $1 \le i < n$. 

\begin{definition}
The extended affine Hecke algebra $\H_n(t)$ is generated by the algebras $\cR[Y_1^{\pm 1},\ldots Y_n^{\pm 1}]$ and $\Hf_n(t)$, with relations: 
\[
T_iY_iT_i=Y_{i+1}, \textrm{for $i=1,\ldots, n-1$},\quad 
T_i Y_j = Y_j T_i, \textrm{ for $j \neq i, i+1$}
\]
\end{definition}
\begin{remark} \label{rem:Y inverse}  A common alternative presentation  imposes instead the relations $T_i Y_i^{-1} T_i = Y_{i+1}^{-1}$.  An isomorphism between the presentations may be given by inverting $Y_i$.
\end{remark}

We denote by $\RYn$ the $\cR$-subalgebra generated by $Y_1^{\pm 1},\ldots Y_n^{\pm 1}$. 
\begin{notation}\label{not:KYn}We denote by $\SymYn = \RYn^{\Sn}$ the space of symmetric Laurent polynomials, which is also the center of $\H_n(t)$.
To simplify notation, we will write the parameters $t$ and $n$ only when required, and otherwise abbreviate $\H_n(t)$ by either $\H_n$ or simply $\H$. Similarly for $\Hf$, $\RY$, and $\SymY$.\end{notation}
 \subsection{$GL$ DAHA} \label{sec: GL DAHA}

\begin{remark}Throughout most of the paper, statements and proofs will be made using $G=\GLN$ notation. We will warn the reader in cases where statements or proofs need modification to be correct for $\SLN$, with the symbol $\SLGL$. 
\end{remark}

\begin{notation}\label{not:groundfieldGL}
    When discussing the $GL$-DAHA we will use the following notation. Fix a $\Q$-algebra $\cR$ which is an integral domain and contains fixed elements $q,t \in \cR^\times$, such that $\frac{t^m-1}{t-1}$ and $\frac{q^m-1}{q-1}$ are invertible for all $m \neq 0$. Let $\cK$ denote the field of fractions of $\cR$. Unless stated otherwise everything will be defined with base ring $\cR$.
\end{notation}

\begin{definition}\label{def:GLDAHA}
The $\GL_n$ double affine Hecke algebra $\HG$ is
the $\cR$-algebra presented by generators:
$$T_0,T_1,\ldots T_{n-1}, \pi^{\pm 1}, Y_1^{\pm 1},\ldots, Y_n^{\pm 1},$$
subject to relations\footnote{As with $\AffSym$, we drop the relations on the 
second line when $n=2$.}:
\begin{align} 
&(T_i-t)(T_i+t^{-1})=0 \quad (i=0,\ldots, n-1),& && \label{HeckeReln}\\
&T_iT_jT_i = T_jT_iT_j\quad (j\equiv i\pm 1 \bmod n),& &T_iT_j = T_jT_i \quad (\textrm{otherwise}),&\label{BraidReln}\\
&\pi T_i\pi^{-1} = T_{i+1} \quad (i=0,\ldots, n-2),& &\pi T_{n-1}\pi^{-1}=T_0,&\nonumber\\
&T_iY_iT_i=Y_{i+1} \quad (i=1,\ldots, n-1),& &T_0Y_nT_0 = q^{-1}Y_1&\nonumber\\
&T_i Y_j = Y_j T_i \quad  (j \not\equiv i, i+1 \bmod n),& &&\nonumber\\
&\pi Y_i\pi^{-1} = Y_{i+1} \quad (i=1,\ldots, n-1),& &\pi Y_{n}\pi^{-1}=
 q^{-1}Y_1&\nonumber
\end{align}
\end{definition}

We recall that $\HG$ has basis $\{  T_w Y^\beta \mid
w \in \AffSym, \beta \in \Z^n\}$ where we identify $\pi$ with $T_\pi$. 

An alternate presentation of $\HG$ includes generators $X_i^{\pm 1}, 1 \le i \le n$, which are related to the generators above via $X_1 = \pi T_{n-1}^{-1} \cdots T_1^{-1}$ and $X_{i+1} = T_i X_i T_i$.  
Similar to above, we may write $\cR[\cX] = \cR[X_1^{\pm 1}, \cdots, X_n^{\pm 1}]$.

\begin{notation} \label{not:AHA-GL}
We denote by $\HY$ and $\HX$, respectively, the subalgebras of $\HH$ generated by $\Hf$ and $Y_i^{\pm}$'s (resp, $X_i^{\pm}$'s). Note $\HX$ is also generated by $\Hf$ and $\pi^{\pm 1}$.  Each subalgebra identifies with $\H$ as an abstract algebra.
\end{notation}

\medskip
\subsection{$SL$ DAHA} \label{sec: SL DAHA}

\begin{notation}\label{not:groundfieldSL}
    When discussing the $SL$ DAHA we will use the following notation. Fix a $\Q$-algebra $\cR$ which is an integral domain containing fixed elements $\Zprod$, $\Sq,t \in \cR^\times$, such that $\frac{\Zprod^m-1}{\Zprod-1}$, $\frac{\Sq^m-1}{\Sq-1}$ and $\frac{t^m-1}{t-1}$ are invertible. Let $\cK$ denote the field of fractions of $\cR$.  Unless otherwise specified, all definitions will have base ring $\cR$.
\end{notation}

\begin{definition}\label{def:SLDAHA}
The $\SL_n$ double affine Hecke algebra $\HS$ is presented by generators:
$$T_0,T_1,\ldots T_{n-1}, \Spi, Z_1,\ldots, Z_n,$$
subject to relations:
\begin{align*} 
&(T_i-t)(T_i+t^{-1})=0 \quad (i=0,\ldots, n-1),& &&\\
&T_iT_jT_i = T_jT_iT_j\quad (j\equiv i\pm 1 \bmod n),& &T_iT_j = T_jT_i \quad (\textrm{otherwise}),&\\
&\Spi T_i = T_{i+1} \Spi \quad (i=0,\ldots, n-2),& &\Spi T_{n-1}=T_0 \Spi,&\\
&T_iZ_iT_i=Z_{i+1} \quad (i=1,\ldots, n-1),&
&T_i Z_j = Z_j T_i \quad  (j \not\equiv i, i+1 \bmod n),&T_0Z_nT_0 = \Sq^{2n}Z_1,&
\\
& \Spi Z_i = \Sq^{-2} Z_{i+1}\Spi \quad (i=1,\ldots, n-1),
&
& \Spi Z_{n}=
 \Sq^{2n-2}Z_1 \Spi,&\\
&Z_1 Z_2 \cdots Z_n = \Zprod, & &\Spi^n=1.&
\end{align*}
\end{definition}

We define the generators $X_i$, and the subalgebras $\cR[\cX]$, $\cR[\cY]$, $\HX$, $\HY$ of $\HS$ in the same way as for $\HG$, noting however that the further relations $Z_1\cdots Z_n = \Zprod$ and $X_1\cdots X_n=1$ now apply.

\begin{notation} $\SLGL$ By abuse of notation $\cR[\cY]$ refers to the $\cR$-algebra generated by the $Z_i$, likewise for the other symbols introduced in Notations \ref{not:KYn} and \ref{not:AHA-GL}; this will be an added convenience for making uniform statements for $G = \GL$ or $\SL$ below.\end{notation}

\begin{remark}
    \label{rem:affine Hecke SLGL} $\SLGL$
    Note that any representation of $\HY \subseteq \HS$ lifts to a representation of $\HY \subseteq \HG$.  Likewise any representation of the latter on which $Y_1 Y_2 \cdots Y_n - \Zprod$ vanishes can be seen as a representation of the former.
\end{remark}

\subsection{Intertwiners} \label{sec:intertwiners}

Cherednik's theory of intertwiners plays an important role in Section \ref{sec:End via intertwiners}.  Readers interested only in Section \ref{sec:End HK via shift} may safely skip the current section.

Recall the standard left action of the extended affine symmetric group $\AffSym$ on $\Z^n$, which depends on an integer $p$:
\begin{align}
    w\cdot (b_1,\ldots, b_n) &= (b_{w^{-1}(1)},\ldots,b_{w^{-1}(n)})\label{eqn:symaction}, \quad \textrm{for $w\in \Sn$},\\
\pi\cdot (b_1,b_2,\cdots b_n) &= (b_n+p,b_1,b_2,\ldots,b_{n-1}),\\
s_0\cdot(b_1,\ldots, b_n) &= (b_n+p,b_2,\ldots,b_{n-1},b_1-p).
\end{align}
\noindent We  take $p=1$ here, but the action is valid for any dilation.
Although we use the symbol $\cdot$ to denote the action, this is not to be confused with the {\em dot} action on weights, which will not (explicitly) appear in this paper.

It is notationally convenient to identify $\Z^n$ with the set of quasi-periodic sequences in $\Z^{\Z}$, as those sequences satisfying $b_{i+mn} = b_i-mp$.  Under this identification, we have:
\begin{align*}
    \pi\cdot (b_1,\ldots, b_n) &= (b_0,b_1,\ldots, b_{n-1})\\
    s_0 \cdot (b_1,\ldots,b_n) &= (b_0,b_2,\ldots, b_{n-1},b_{n+1}),
\end{align*}
and so \eqref{eqn:symaction} now holds for all $w \in \AffSym$, not just for $w \in \Sn$.  We will call elements of $\AffSym$ affine permutations. 

Given an element $\beta \in \Z^n$,  \defterm{translation} by $\beta$ is the affine permutation $\tr(\beta)$ sending $i \mapsto i+ n\beta_{i \bmod n}$.  The translation then acts on $\Z^n$ via
$\tr(\beta) \cdot (b_1,\ldots, b_n) = (b_1 + p\beta_1,\ldots, b_n +  p\beta_n) = (b_1 + \beta_1,\ldots, b_n +  \beta_n),$ given we have taken $p=1$ above. The translations form an abelian  normal subgroup of $\AffSym$. 

\medskip

Similarly
we may define an action of $\AffSym$ on $\cR^n$. The subgroup $\Sym_n$ acts by permuting coordinates as in \eqref{eqn:symaction},
but now:
\begin{align*}
\pi\cdot (b_1,b_2,\cdots b_n) &= (q b_n,b_1,b_2,\ldots,b_{n-1}),\\
s_0\cdot(b_1,\ldots, b_n) &= (q b_n,b_2,\ldots,b_{n-1},q^{-1} b_1),
\end{align*}
and so we identify $b_{i+mn} = q^{-m} b_i$.

When quotienting $\AffSym$ by the subgroup generated by $\pi^n$ we have to modify $\SLGL$ the action of $\pi$ to $\Spi$ as follows
\begin{align} \label{eq:modify pi for SL}
\Spi\cdot (b_1,b_2,\cdots b_n) &= (\Sq^{-2n}\Sq^2 b_n,\Sq^2 b_1,\Sq^2 b_2,\ldots,\Sq^2 b_{n-1}).
\end{align}
and then can compute $s_0 = \Spi^{-1} s_1 \Spi$.

Let us extend the subscripts of $T$, $Y$, and $Z$ 
 to take arbitrary integer value, by letting
\begin{align}
    T_{j+mn} &:= T_j, \textrm{ for $j\in \{0,\ldots, n-1\}$ and } m\in \mathbb{Z} \\
    Y_{i+mn} &:= q^{-m}Y_i,
    \textrm{ for $i\in \{1,\ldots, n\}$ and $m\in \mathbb{Z}$.}\\
    Z_{i+mn} &:= \Sq^{2nm}Z_i,
    \textrm{ for $i\in \{1,\ldots, n\}$ and $m\in \mathbb{Z}$.}
\end{align}

$\SLGL$ The following operators and relations are written in $G = \GL$ notation, but obvious modifications  can adapt them to $G=\SL$.

\begin{definition}\label{def:loc}
We let $\cR[\Yloc]$ and $\HHloc$, respectively, denote the Ore localization of $\RY$ (respectively, $\HH$) at the set of $\{\fij{i}{j}\}$ for integers $i,j \in \Z$ with $i \not \equiv j \bmod n$, where
\[
\fij{i}{j} := t Y_i - t^{-1} Y_j
.\]
\end{definition}

\begin{definition}\label{def:intertwiners}
For each integer $i$, we recall \defterm{Cherednik's intertwiners}:
    \[\phii := T_iY_i-Y_iT_i = T(Y_i-Y_{i+1}) + (t-t^{-1})Y_{i+1}.\]
    \item The \defterm {renormalized intertwiners} are
    \[\nui: =\phii(\fij{i}{i+1})^{-1}\in \HHloc.\]
\end{definition}

We recall the following well-known intertwining relations for Cherednik's and the renormalized intertwiners:
\begin{alignat}{4}
    &Y_j\phii & &= \phii Y_{s_i(j)}&
    \qquad 
    &\pi \phii & &= \phij{i+1}\pi, \label{eq:piphi}\\ 
     &Y_j\nui & &= \nui Y_{s_i(j)}&
     \qquad    &\pi\nui & &=\nuj{i+1}\pi.  \label{eq:pinu}
\end{alignat}

We additionally recall the following  braid and almost-quadratic relations for Cherednik's intertwiners:
    \begin{align}
    \phii\phij{j}\phii &= \phij{j}\phii\phij{j},& \textrm{if $i\equiv j\pm 1 \bmod n$}\label{eq:phibraid1}\\
    \phii\phij{j} &= \phij{j}\phii,& \textrm{if $i\not\equiv j\pm 1 \bmod n$}\label{eq:phibraid2}\\
    \phii^2 &= \fij{i}{i+1}\fij{i+1}{i}.
    \end{align}

These easily imply braid and quadratic relations for renormalized intertwiners:
    \begin{align}
    \nui\nuj{j}\nui &= \nuj{j}\nui\nuj{j},& \textrm{if $i\equiv j\pm 1 \bmod n$}\label{eq:nubraid1}\\
    \nui\nuj{j} &= \nuj{j}\nui,& \textrm{if $i\not\equiv j\pm 1 \bmod n$}\label{eq:nubraid2}\\
    \nui^2 &=1.
    \end{align}
Finally, although we will not make use of them, we record the following mixed braid relations: 
    \begin{align*}
    \nui T_{j}\nui &= \nuj{j}T_i\nuj{j},& &\textrm{if $i\equiv j\pm 1\bmod n$}\\
    \nui\nuj{j}T_i&=T_{j}\nui\nuj{j},& &\textrm{if $i \equiv j\pm 1 \bmod n$}\\
     T_i\nuj{j} &= \nuj{j}T_i,& &\textrm{if $i\not\equiv j-1,j,j+ 1 \bmod n$.}\label{eq:nubraid2}
\end{align*}
$\SLGL$ For the $\SL$ version of the above definitions and relations, using $Z_i$ in place $Y_i$, the only relation we need modify is equation \eqref{eq:piphi} to $\Spi \phii = \Sq^{-2} \phij{i+1}\Spi$.
 \medskip

Let us also adopt the convention that $\phij{\pi} = \pi$ and $\nuj{\pi} = \pi$.  Given an affine permutation $w$ with reduced word decomposition $s_{i_1}\ldots s_{i_m}$, we define $\phij{w} = \phij{i_1}\cdots \phij{i_m}$, and $\nuj{w} = \nuj{i_1}\cdots \nuj{i_m}$.  The braid relations \eqref{eq:phibraid1},\eqref{eq:phibraid2},\eqref{eq:nubraid1}, as well as \eqref{eq:piphi} and \eqref{eq:pinu},   ensure that $\phij{w}$ and $\nuj{w}$ are well-defined independent of the choice of reduced word.

Given an  affine permutation $w$, we let  $\Inv(w)$ denote the set of its \defterm{inversions}
\[
\Inv(w) := \{ (i,j) \in \{1,\ldots n\} \times \Z,  \,\, | \,\, i<j \textrm{ and } w(i) > w(j)  \}.
\]
Then we have the following formula relating $\phij{w}$ and $\nuj{w}$:
\begin{equation}\label{eqn:phivsnu}
    \phij{w} = \nuj{s_{i_1}}\fij{i_1}{i_1+1}^{-1}\cdots \nuj{s_{i_m}}\fij{i_m}{i_m+1}^{-1} = \nuj{w}\cdot q^{\alpha}\prod_{(i,j)\in \Inv(w)}\fij{i}{j}
\end{equation}
 which follows immediately from the intertwining relations. The exponent $\alpha$ on $q^\alpha$ arises from applying the relation $\fij{i}{j} = q \fij{i+n}{j+n}$ in order to enforce our convention that $1 \le i \le n$ for $(i,j) \in \Inv(w)$.  
 \begin{example}
 For example, let $w= s_1 s_2 s_0 s_1 s_2 s_0 s_1 s_2 \in \AffSymn{3}$.  While the natural subscripts $(i,j)$ on the product $\prod \fij{i}{j}$ from \eqref{eqn:phivsnu} are 
 \[
 \{(-2,9), (-1,9), (-2,6), (-1,6), (1,6), (-1,3), (1,3), (2,3) \}\]  we define  \[\Inv(w) = \{ (1,12),(2,12),(1,9),(2,9),  (1,6),(2,6),  (1,3), (2,3)  \}.\]  
Thus $\alpha = 5$ in \eqref{eqn:phivsnu} for this $w$.
 \end{example}

\subsection{Induced modules}\label{sec:induced modules}  

For the remainder of this section, we work over $\K$. In other words, we take $\cR$ to already be a field and so $\cR = \K$.

\begin{definition} \label{def:SH}
Let $\SHn$ denote the subalgebra of $\HYn$ generated by $\Hfn$ and $\SymYn$. 
\end{definition}
We note the decomposition $\SHn = \SymYn \otimes_{\cK} \Hfn$ as algebras.  Just as $\cK[\Yn]$ is a free $\SymYn$-module of rank $n!$,  $\HYn$ is a free $\SHn$-module of rank $n!$.  Recall $\HYn$ is also a free $\cK[\Yn]$-module of rank $n!$. When $n$ is understood, we sometimes merely write $\SH$,  as we may write $\Y$ for $\Yn$. 

\begin{notation}
Write $\atup= (a_1, \cdots, a_n)\in (\cK^\times)^n$ for the one dimensional $\cK[\Y]$-module on which all $Y_i - a_i$ vanish. Let  $\vv$ be a basis of this one-dimensional vector space so that $\atup = \cK \vv$.
Write $\{\atup\}$ for the one-dimensional $\SymY$-module obtained as the restriction of $\atup$. In other words, upon which  $\sum_{i=1}^n (Y_i^m - a_i^m)$ vanish for all $m$. 
Alternatively, one may parameterize the module by the $\Sn$-orbit of $\atup$, which is the unordered collection or multiset of the $a_i$, hence our notation.
Write $\asgn$ for the $\SH$-module, 1-dimensional over $\cK$, on which $\SymY$ acts as $\{ \atup \}$
and $T_i + t^{-1}$ vanishes for all $1 \le i < n$. 
Let us write  $\uu$ for a basis of this one-dimensional vector space so that $\asgn  = \cK \uu$.
\end{notation}
\begin{remark}
The above construction makes sense for any character of $S(\Y)$, and it is just for convenience that we have chosen one which is the restriction of a character for $\Y$.
\end{remark}

$\SLGL$ In the case of $G=\SLn$, we will also require $\prod_i a_i = \Zprod$, and replace $t$ with $t^{1/N}$ in defining descending, but no other modifications are required.  See also Remark \ref{rem:SLGL affsym} below.
\begin{definition}\label{def: reverse right order}
  We call an $n$-tuple $\atup = (a_1, \cdots, a_n) \in (\cK^\times)^n$ \defterm{descending}
  if $a_i/a_j = t^{2z}$ with $z \in \Z$  and $i< j$ implies $z \ge 0$.   
\end{definition}

\begin{remark}\label{rem:reverse}
If $\atup$ is descending 
then $\fij{i}{j} \otimes \vv \neq 0$ for all $1 \le i < j \le n$, i.e., $t a_i - t^{-1} a_j \neq 0$ when $i < j$.
\end{remark}

Then we have the following isomorphism of $n!$-dimensional $\HYn$-modules. To lighten notation below we merely write $\Y$ for $\KY$. 
\begin{theorem}\label{thm:Ind SH}
Suppose that $\atup$ is 
descending. 
Then the map
\begin{align*}
    \Ind_{\SH}^{\HY} \asgn &\to \Ind_\Y^{\HY} \atup \\
    1 \otimes \uu &\mapsto \ek{n} \otimes \vv
\end{align*} 
is an isomorphism. 
\end{theorem}
\begin{proof}
Observe that $\sum_{i=1}^n (Y_i^m - a_i^m) T_j \otimes \vv = T_j \sum_{i=1}^n (Y_i^m - a_i^m) \otimes \vv = 0 $ for $1 \le j < n$ so 
the map 
$$A:  \Ind_{\SH}^{\HY} \asgn \to \Ind_{\Y}^{\HY} \atup$$
determined by $A(1 \ot u) = \ek{n} \ot \vv$ is a well-defined $\HY$-module map.
Due to the well-known equality
\[\dim_\K \Ind_{\SH}^{\HY} \asgn = \dim_\K \Ind_{\Y}^{\HY} \atup = n!,\] it suffices to show $A$ is surjective.

An easy calculation yields
\begin{gather*}
    (Y_i - a_{i+1}) \frac{1}{t+t^{-1}} (t-T_i) \ot \vv
    = 1 \ot \frac{t a_i - t^{-1} a_{i+1}}{t+t^{-1}} \vv
\end{gather*}
which agrees with $\frac{1}{t+t^{-1}} \fij{i}{i+1} \ot \vv= \frac{t}{t^2+1} \fij{i}{i+1} \ot \vv$. In particular, by Remark \ref{rem:reverse} and our assumption that $\atup$ is descending, this is nonzero. 
This is the heart of the $n=2$ case. 
More generally, let $g(Y) = \prod\nolimits_{1 \le i < j \le n}(Y_i - a_j).$
One can show
\begin{gather*}
    g(Y) \ek{n} \ot \vv 
    = \frac{t^{\binom{n}{2}}}{[n]_{t^2} !}  \prod_{1 \le i < j \le n} (t a_i - t^{-1} a_j) \ot \vv
    =    \frac{t^{\binom{n}{2}}}{[n]_{t^2} !} \prod_{1 \le i < j \le n} \fij{i}{j} \ot \vv.
\end{gather*}

As $\atup$ is descending, 
the above expression evaluates 
at $\atup$ to a nonzero scalar $\alpha$ times $1 \ot \vv$, and so $\ek{n} \ot \vv$ generates $\Ind_{\Y}^{\HY} \atup$ as an $\HY$-module. Thus $A$ is a surjection.  
This completes the proof.

Observe, as $g(Y)\ot \uu$ is a $\Y$-weight vector of weight $\atup$, the map 
\begin{align*}
    B: \Ind_{\Y}^{\HY} \atup &\to \Ind_{\SH}^{\HY} \asgn \\
    1 \ot \vv   &\mapsto \frac {1}{\alpha} g(Y) \ot \uu
\end{align*}
is inverse to the map $A$.
\end{proof}

\section{Endomorphisms of $\HKuniv$ via elliptic Schur--Weyl duality}\label{sec:EndHK}

In this section, we recall the elliptic  Schur-Weyl duality functor from \cite{Jordan2008}, and observe that it intertwines the action of the algebra $\cB^G \cong \Oq^G$ of Casimirs, as appears in the definition of $\HKchi$, with the natural action by the distinguished subalgebra $\SymY\subset \DAHA{N}{N}$.  Exploiting this compatibility we prove several fundamental properties of Hotta--Kashiwara modules.

\begin{notation}\label{not:parameters-all-together}
In this section, we need to unite our many parameters $t, \q, q, \Sq, \Zprod$ as mentioned in Remark \ref{rem: various q} into a single ground ring.   We fix a natural number $N$ and we let $\cR$ be local ring $\Q[\q^{ \frac 1N}]_{(\q-1)}$, and so that $\K=\Q(\q^{\frac{1}{N}})$ is its field of fractions, and we let $\kappa=\cR/(\q-1)$ denote the residue field at $\q=1$.   The other parameters appear as follows:
\begin{itemize}
\item The quantum group parameter is $\q = (\q^{\frac{1}{N}})^N$.
\item The quadratic parameter in Hecke algebras is $t = \q= (\q^{\frac{1}{N}})^N$.  Taking $t= \q $ ensures compatibility with Schur-Weyl duality.
\item The loop parameter for the rank $n$ $\GL$ DAHA $\HG$ is specialized $q = t^{-2n/N} = \q^{-2n/N}$.  Henceforward, we lighten notation and write $\DAHA{N}{n}$ for this specialization.
\item The loop parameter for the rank $n$ $\SL$ DAHA  $\HS$ is specialized $\Sq = t^{1/N} =\q^{1/N}$, and further we take $\Zprod = \Sq^{n(n-N^2)} = t^{n(n/N-N)}$. 
Henceforward, we lighten notation and also write $\DAHA{N}{n}$ for this specialization.
\end{itemize}
We will have need to discuss various $\cR$-modules, denoted $\Rform{M}$, their localization to $\cK$, denoted $\Kform{M}$, and their specialization at $\q=1$ denoted $\qform{M}$.
 
\end{notation}

\subsection{The elliptic Schur--Weyl duality functor}\label{sec:psi is an isom}
Let us fix $G=\GLN$ for some integer $N$.  We will discuss the modification of statements and their proofs for $\SLN$ at the end of this section, indicated by $\SLGL$.   Let $M$ be a strongly equivariant $\Dq$-module, and recall $V$ denotes the defining representation of $U_\q(\mathfrak{g})$.  For each non-negative integer $n$ we have a functor \cite{Jordan2008},
\begin{align*}
    F_n: \Dqstr{G} &\longrightarrow \DAHA{N}{n}\modu,\\
      M \quad &\mapsto  \quad \Hom_{U_\q(\g)}((\wedgeN{N})^{\otimes n/N}, V^{\ot n}\ot M).
\end{align*}

\begin{remark}
Unwinding the definitions, we see that
\[
F_n(M) \cong \Hom_{\Dqstr{G}}(\Dist((V^{\ast})^{\otimes n} \otimes (\wedgeN{N})^{\otimes n/N}), M)
\]
(see Remark \ref{rem:Dist}).
\end{remark}

In the case $G=\GLN$, strong equivariance implies $F_n(M) = 0$ unless $n/N \in \ZZ_{> 0}$. For $G=\SLN$ we have no such restriction as $\wedgeN{N}$ is  trivial.
\begin{notation}\label{not:z1}
Recall from \eqref{eq:z} that $z\in V^{\ot N}$ denotes the sign-like element.
Denoting by $1$ the distinguished cyclic generator of $\HKuniv$, we will regard $z\otimes 1$ as an element of $F_N(\HKuniv)$  which by Proposition \ref{prop: HK chi not zero} is nonzero.
\end{notation}

\begin{notation}
We will re-use the notation $c_k$ from Notation \ref{not:ck-def} now to refer to the image of the quantum Casimir elements in $\cB^G$ under the embedding $i_{\cB}$ or $\iBt$.
\end{notation}

\begin{remark}
We note that the elements $c_k \in \Dq^G$ are not central in $\Dq$.  However, they commute with the image of the moment map: indeed we have that for any $\ell\in \Oq$ and for $y\in\Dq^G$, we have, in Sweedler notation,
\[\ad(\ell) \,  y = (\Rosso(\ell)_{(1)} \rhd y) \ad(\ell_{(2)}) =  y \, \ad(\ell).\]
\end{remark}

It will be convenient to re-express the functor $F_N$ in terms of a Schur-Weyl duality homomorphism from the double affine Hecke algebra to a certain $\Hom$ space in the category of strongly equivariant $\Dq$-modules.  We consider three homomorphisms, $\SWdaff$, $\SWeN$, and $\eSWeN$, the latter two of which we consider in this paper. 

We have  the  following three $\cR$-module homomorphisms, of which the first  and third \eqref{eq:SW}, \eqref{eq:eSWe} are algebra  homomorphisms and the second \eqref{eq:SWe} is a $\DAHA{N}{N}$-module homomorphism. 
\begin{align} 
    \SWdaff: \DAHA{N}{N} &\longrightarrow \Hom(\Dist(V^{\ot N}),\Dist(V^{\ot N}))=F_N(\Dist(V^{\ot N})) \label{eq:SW} \\
    \SWeN: \DAHA{N}{N}\cdot \eN &\longrightarrow \Hom(\Dist(V^{\ot N}),\HKuniv) = F_N(\HKuniv), \label{eq:SWe}\\
    \eSWeN: \eN\cdot\DAHA{N}{N}\cdot \eN &\longrightarrow \Hom(\HKuniv,\HKuniv) = \HKuniv^G \label{eq:eSWe}
\end{align}

In this section, we prove that $\eSWeN$ is an isomorphism over $\K$.  In the forthcoming paper \cite{GJVY} we prove moreover that $\SWdaff$ and $\SWeN$ are also isomorphisms over $\K$.

As with other notation, we may lighten this to $\SWe$ and $\eSWe$ when $N$ is understood. 
\subsection{Compatibility of $\SWdaff$ with bimodule actions} We equip $\DAHA{N}{N}\cdot \eN$ with the structure of a $\DAHA{N}{N}$-$\SymY$-bimodule, with $\DAHA{N}{N}$ acting naturally on the left, and $\SymY$ multiplying on the right, noting that $\eN$, commutes with all elements of $\SymY$.  Let us fix an identification 
\begin{align}
\label{eq:SymB}\SymY &\cong \cB^G  \\
s_{(1^k)} &\mapsto \cas{k} \notag
\end{align}
sending the $k$th elementary symmetric polynomial $s_{(1^k)}$  to
the invariant $\cas{k}$. This equips $\HKuniv$ with the structure of $\Dq$-$\SymY$-bimodule, with $\SymY$ acting by right multiplication through the inclusion $\iBt: \cB^G\subseteq \Dq^G$.  

In the same way we equip $F_N(\HKuniv)$ with the structure of a $\DAHA{N}{N}$-$\SymY$ bimodule, with $\DAHA{N}{N}$ acting as constructed in \cite{Jordan2008}\ and $\SymY$ acting by right multiplication as in \eqref{eq:SymB}.

\begin{proposition} \label{prop:Y and c}
The map $\SWeN$ is a homomorphism of $\DAHA{N}{N}$-$\SymY$-bimodules.
\end{proposition}

\begin{proof}
The map $\SWeN$  of \eqref{eq:SWe} is an $\DAHA{N}{N}$-module map by construction. Hence it only remains to show that $\SWeN$ intertwines the two right $\SymY$-actions.  
Since $\DAHA{N}{N} \cdot \eN$ is cyclic as an $\DAHA{N}{N}$-module, it suffices to equate the two $\SymY$ actions on the cyclic generators $\eN$ and $z \otimes 1 = \SWeN(\eN)$.

We recall from \cite{Jordan2008} that the action of $\cR[\Y] \subseteq \DAHA{N}{N}$ on $F_N(M)$ for any $\Dq$-module $M$ can be expressed using graphical calculus for braided tensor categories. 
In particular, we have:

    \begin{equation}
   \begin{tikzpicture}[line width=2pt, xscale=0.8]
\node[left] at (2,4) {$(Y_1 \cdots Y_k)(z \ot 1)$ \qquad $=$};
        \node[above] at (4,4) {$V^{\otimes N-k}$};
        \node[above] at (5,4) {${\otimes }$};
        \node[above] at (6,4) {$V^{\otimes k}$};
        \node[above] at (7,4) {${\otimes }$};
        \node[above] at (8,3.95) {$(V^{\otimes k})^\ast$};
        \node[above] at (9,4) {${\otimes }$};
        \node[above] at (10,4) {$V^{\otimes k}$};
        
        \draw (6,4) .. controls (6.3,3.5) and (7.7,3.5) .. (8,4);
        \draw (5,5) -- (5,6.5);
        \node[above] at (5,6.5) {$V^{\otimes N}$};
        \draw (9,5) -- (9,5.5);
        \draw (9,6.2) -- (9,6.5);
        \node[above] at (9,5.35) {\framebox{$i_\cB$}}; 
        \node[above] at (9,6.5) {$\cB$};
        \draw (4,3) -- (4,4);
        \node[above] at (4,2) {$V^{\otimes N-k}$};
        \node[above] at (5,2) {${\otimes }$};
        \node[above] at (6,2) {$V^{\otimes k}$};
        \draw (5,2) -- (5,1.6);
        \node[above] at (5,0.8) {$\wedgeN{N}$};
        \Under[6,2.7][10,4];
\end{tikzpicture}
\noindent
\end{equation}
Because $\eN \ek{k} = \eN = \ek{k} \eN$ for $k \le N$ we have  
\begin{equation}
    \begin{tikzpicture}[line width=2pt, xscale=0.8]
\node[left] at (2,4) {$\eN (Y_1 \cdots Y_k) \eN \cdot (z\otimes 1)$ \qquad $=$};
        \node[above] at (4,4) {$\wedgeN{N-k}$};
        \node[above] at (5,4) {${\otimes }$};
        \node[above] at (6,4) {$\wedgeN{k}$};
        \node[above] at (7,4) {${\otimes }$};
        \node[above] at (8,3.95) {$(\wedgeN{ k})^\ast$};
        \node[above] at (9,4) {${\otimes }$};
        \node[above] at (10,4) {$\wedgeN{ k}$};
        \draw (6,4) .. controls (6.3,3.5) and (7.7,3.5) .. (8,4);
        \draw (5,5) -- (5,6.5);
        \node[above] at (5,6.5) {$\wedgeN{ N}$};
        \draw (9,5) -- (9,5.5);
        \draw (9,6.2) -- (9,6.5);
        \node[above] at (9,5.5) {$i_\cB$};
        \node[above] at (9,6.5) {$\cB$};
        \draw (4,4) -- (4,3.5);
        \node[above] at (4,3) {$\id_{\wedgeN{N-k}}$};
        \draw (4,3) -- (4,2.7);
        \node[above] at (4,2) {$\wedgeN{N-k}$};
        \node[above] at (5,2) {${\otimes }$};
        \node[above] at (6,2) {$\wedgeN{ k}$};
        \draw (5,2) -- (5,1.6);
        \node[above] at (5,0.8) {$\wedgeN{ N}$};
        \Under[6,2.7][10,4];
\end{tikzpicture}
\label{eqn:eYse}
\end{equation}

\noindent
which is an element of the one-dimensional vector space, $\Hom_{\Uq}(\wedgeN{N}, \wedgeN{N} \otimes C(\wedgeN{k}))$ where $C(\wedgeN{k}) \cong (\wedgeN{k} )^\ast \otimes \wedgeN{k} \subseteq \cB \cong \Oq$ denotes the $\cR$-submodule of matrix coefficients 
of $\wedgeN k $ under the Peter-Weyl decomposition. 
Another such homomorphism is given by multiplication by $\cas{k}$ 
and is depicted in graphical calculus by 

\begin{equation} 
    \begin{tikzpicture}[line width=2pt, xscale=0.8]
\node[left] at (2,3) {$\cas{k} (z \ot 1)$ \, $=$};
\node[above] at (4,4) {$\wedgeN{N}$};
        \node[above] at (5,4) {${\otimes }$};
        \node[above] at (8,4) {$\wedgeN{k}$};
        \node[above] at (7,4) {${\otimes }$};
        \node[above] at (6,3.95) {$(\wedgeN{ k})^\ast$};
        \draw (4,2) -- (4,4);
        \node[below] at (4,2) {$\wedgeN{N}$};
        \begin{knot}[
consider self intersections,
clip radius=8pt,
clip width=5
]
\strand [ultra thick]
(6,4)
to [out=down, in=up]
(7.5,3)
to [out=down, in=down]
(6.6,3)
to [out=up, in=down]
(8,4);
\end{knot}
\end{tikzpicture}
 \label{eqn:c_k}
 \end{equation}

\noindent
It follows that $\cas{k}  (z\otimes 1) = \lambda  \eN  Y_1 Y_2 \cdots Y_k  \eN (z\otimes 1)$ for some $\lambda \in \cR$. 
In order to determine the scalar $\lambda$, we apply $\id_{\wedgeN{N}} \otimes \epsilon_\cB$ to both sides.

Applying $\id_{\wedgeN{N}} \otimes \epsilon_\cB$ to Equation \eqref{eqn:eYse} gives simply $z\otimes 1$.  On the other hand, applying $\id_{\wedgeN{N}} \otimes \epsilon_\cB$ to Equation \eqref{eqn:c_k} gives scalar multiplication by ${\genfrac{[}{]}{0pt}{}{N}{k}}_{\q^{-2}}$. 
 Hence $\lambda = {\genfrac{[}{]}{0pt}{}{N}{k}}_{\q^{-2}}$, and so $\eN \cas{k} \eN = {\genfrac{[}{]}{0pt}{}{N}{k}}_{\q^{-2}} \eN Y_1 \cdots Y_k \eN$. We have the same proportion in $\DAHA{N}{N}$, i.e., $\eN s_{(1^k)} \eN = {\genfrac{[}{]}{0pt}{}{N}{k}}_{\q^{-2}}  \eN Y_1 \cdots Y_k \eN$, and so the $\SymY$ actions precisely coincide, as claimed in \eqref{eq:SymB}. 
\end{proof}

As an aside, observe the scalars computed above satisfy ${\genfrac{[}{]}{0pt}{}{N}{k}}_{\q^{-2}} =\q^{k(k-N)}\dim_\q(\wedgeN{k})$ in the conventions of this paper.  

We record the following corollary, which will appear later in proof of Corollary \ref{cor:qrho}.

\begin{corollary} \label{cor:epsilon as qrho}
Under the identification $\cB^G\cong \SymY$, the restriction to $\cB^G$ of the trivial character $\epsilon:\cB\to\K$ coincides with the restriction to $\SymY$ of the character $\q^{-2\rho}$.    
\end{corollary}

 \begin{proof}
 In Proposition \ref{prop:Y and c} we have shown $\cas{k}$ acts on $z \otimes 1 \in \wedgeN{N} \otimes \HKe^G$ as \[{\genfrac{[}{]}{0pt}{}{N}{k}}_{\q^{-2}} = s_{(1^k)}(1, \q^{-2}, \q^{-4},\ldots, \q^{2-2N}).\]    
 \end{proof}
 
\renewcommand{\DN}{\HH}

\subsection{Polynomial forms} \label{sec:positive HK}
In the proof of Theorem \ref{thm:SWsurjective}, we will apply Nakayama's lemma arguments.  These arguments are applicable to finitely generated $\cR$-modules, however the modules to which we apply them are not finitely generated.  In the $\GL$ setting, we have a natural grading, and in the $\SL$ setting we have a natural filtration, however even the graded (resp. filtered) parts are not finitely generated.

We will define ``polynomial" forms of the $\cR$-modules $\Dq$, $\HKuniv$, and $\DN$ with the property that the graded parts are finitely generated, and the entire module is obtained by localizing (in the $\GL$ case) or specialising (in the $\SL$ case) certain explicit quantum determinants.  This manoeuvre will allow us to apply Nakayama's lemma.

\begin{definition}\label{defn:REA+} The twisted reflection equation algebra of type $\GLN$, denoted $\overline{\cO_\q(\MatN)}$, is the algebra generated by symbols $\bar \ellgen^i_j$, for $i,j = 1,\ldots N$, subject to the relations,
 $$R_{21}^{-1}\bar{L}_1R_{12}^{-1}\bar{L}_2 =  \bar{L}_2R^{-1}_{21}\bar{L}_1R^{-1}_{12},$$
where $\bar{L} := \sum_{i,j} \bar{\ellgen}^i_j E^j_i$ is a matrix with entries the generators $\bar{\ellgen}^i_j$.
\end{definition}

Note that $\overline{\cO_\q(\MatN)} \simeq \cO_{\q^{-1}}(\MatN)$, as one obtains the inverse R-matrix by inverting $\q$. 

We note that by construction $\cO_\q(\MatN)$ and $\overline{\cO_\q(\MatN)}$ are positively graded with finite-dimensional graded pieces, a property which does not survive to their common localization to $\OqG{\GLN}$ nor their common quotient to $\OqG{\SLN}$. 

In order to give a polynomial version of $\Dq$, we need to introduce the following modifications.

\begin{definition} \label{def:Dq+} For $G=\GLN$, or $\SLN$, the algebra $\Dq^+$ is the twisted tensor product,
\begin{gather}
\label{eq-Dq+=Oq'-Oq}
\Dq^+ = \overline{\cO_\q(G)} \mathbin{\widetilde{\ot}} \cO_\q(G).
\end{gather}
Denoting $\bar a^i_j$ and $ b^i_j$ to be the generators of the first and second factors, the cross relations are given in matrix form by:
\begin{align}\bar A_1 R_{21}^{-1}  B_2 R_{21} &= R_{12} B_2 R_{21}\bar A_1, & \textrm{if $G=\GLN$} \label{eqn:Dq+Relns}\\
\bar A_1 R_{21}^{-1}  B_2R_{21} &= R_{12} B_2R_{21} \bar A_1 \q^{-2/N}, & \textrm{if $G=\SLN$}\label{eqn:Dq+Relns2}
\end{align}
where $\bar A= \sum_{i,j}\bar a^i_jE^j_i$ and $B=\sum_{i,j} b^i_j E^j_i$.
\end{definition}

Recall that $\HKuniv$ is naturally a quotient of $\Dq$ by a certain left ideal $J$. 
It will be convenient to use the matrix notation from Proposition \ref{prop:containment} to describe that ideal and motivate the definition of its polynomial version.  Since the matrices $A$ and $B$ 
are invertible, we may rewrite $J= \Dq\cdot C(\ad) = \Dq\cdot C''(\ad)$,
where 
\[
C''(\ad) =  \{\textrm{matrix coefficients of the matrix } A^{-1}B- BA^{-1}\}.
\]

\begin{definition}
For $G=\GLN$ or $\SLN$, the $\Dq^+$-module $\HKuniv^+$ is the quotient,
\[
\HKuniv^+ = \Dq^+/\Dq^+\cdot \bar{C}(\ad),
\]
where 
\[
\bar C(\ad) =  \{\textrm{matrix coefficients of the matrix } \bar AB- B\bar A\}.
\]
\end{definition}

By construction we have homomorphisms $\Dq^+ \to \Dq$ given by $\bar{A},B\mapsto A^{-1},B$. Note that for $G=\GLN$ this map is a localization whereas for $G=\SLN$ it is a quotient.  In either case we have an isomorphism,
\[
\HKuniv = \Dq\otimes_{\Dq^+}\HKuniv^+.
\]

We note for use in Theorem \ref{thm:SWsurjective} that $\HKuniv^+$ is  a graded $\cR$-module all of whose graded components are finitely generated over $\cR$. 
We now turn to the polynomial submodules in the DAHA setting, where definitions are more straightforward.

Let $\cR[\cY]^+ = \cR[Y_1, \cdots, Y_n] $ and  
$\cR[\cX]^+ = \cR[X_1^{-1}, \cdots, X_n^{-1}] $.  It is convenient to renotate $\bar X_i := X_i^{-1}$. We denote by $\HG^+$ the $\cR$-subalgebra of $\HG$ generated by $\Hf$, $\cR[\cY]^+$ and $\cR[\cX]^+$. 
By construction, we have that $\HG$ is the localization, 
\[\HG=\HG^+[(Y_1\dots Y_n)^{-1}, (\bar X_1\dots \bar X_n)^{-1}]\] at the elements $(Y_1\dots Y_n)$, $(\bar X_1\dots \bar X_n)$.

 We endow  $\HG^+$ with a grading of $\cR$-modules by setting $\deg(\bar X_i)=\deg(Y_i)=1$ for all $i=1,\ldots n$, and $\deg(T_j)=0$ for $j=1\ldots n-1$, and note that the resulting homogeneous components are finite rank and free over $\cR$.

\begin{remark}
Alternatively $\HG^+$ is the subalgebra generated by 
negative powers of $\pi$,  $\cR[\Y^+]$, and $\Hf$,
noting $(\bar X_1\cdots \bar  X_n) = \pi^{-n}$. With respect to the grading above, we have $\deg(\pi^{-1})=1$.
\end{remark}

We define $\HS^+$ to be the algebra generated by $\Hfn$, $Z_i$, $\Spi^{-1}$, with relations obtained by modifying each relation in Definition \ref{def:SLDAHA} (by multiplying on left and right by $\Spi^{-1}$ as needed) to involve only $\Spi^{-1}$. However, we do not include the final two relations $Z_1\cdots Z_n =\Zprod$ and $\Spi^{-n}=1$ (modified from the original relation $\Spi^{n}=1$).  We denote the corresponding subalgebras of $\HS^+$ by $\cR[\cY]^+ := \cR[Z_1, \cdots, Z_n]$ and $\cR[\cX]^+ := \cR[\bar X_1, \cdots,  \bar X_n]$, where $\bar X_1 = T_1 \cdots T_{n-1} \Spi^{-1}$ and $\bar X_{i+1} = T_i^{-1} \bar X_{i} T_i^{-1} $.  We note that $\HS^+$ surjects to $\HS$ with kernel (generated by the omitted relations), rather than embedding as a subalgebra.

Recall that we have the Schur-Weyl duality map \eqref{eq:eSWe}
\[
\eSWeN \colon \eN\DN\eN \to \HKuniv^G.
\]
In the case $G=\GLN$ we note that this map takes the positive subspace $\eN \DN^+ \eN$ to $(\HKuniv^+)^G$.  In the $\SLN$ case, we note instead that the generators-and-relations construction of $\SWdaff\eN$ in \cite{J2008} 
immediately lifts to a map from $\eN\DN^+\eN$ to $(\HKuniv^+)^G$.  We formulate these observations more precisely in the following proposition.

\begin{proposition}\label{prop:plusforms}
Each of $\eN \DN^+\eN$ and $(\HKuniv^+)^G$ are positively graded
$\cR$-modules with finitely generated graded components, and we have a map of graded $\cR$-modules 
    \[
    (\eSWeN)^+ \colon \eN\DN^+\eN \to (\HKuniv^+)^G
    \]
    such that we recover the map
    \[
    \eSWeN \colon \eN \DN \eN \to (\HKuniv)^G
    \]
    by localizing $ \bar {X}_1\ldots  \bar {X}_N Y_1\ldots Y_N$ on the source and $\operatorname{det}_{\q^{-1}}(\bar A)\detq(B)$ on the target in the $\GLN$ case, and by setting $ \bar {X}_1\ldots  \bar  {X}_N Y_1\ldots Y_N=1$ and $\operatorname{det}_{\q^{-1}}(\bar A)\detq(B)=1$ in the  $\SLN$ case.
\end{proposition}

\subsection{Proof of Theorem \ref{mainthm:HKuniv}}
We are now ready to prove the main result of this section, which is Theorem \ref{mainthm:HKuniv} from the introduction.
We now fix $n=N$ throughout, and suppress these from the notation for $\DAHA{N}{n}$, 
$\SH_N$, $\eN$, etc.  
Recall as in Notation \ref{not:parameters-all-together} we have specialized $q=t^{-2} = \q^{-2}$ and $\Sq=t^{1/N} = \q^{1/N}$.

\renewcommand{\eN}{\ee}

\begin{theorem} \label{thm:SWinj}
We have the following:
\begin{enumerate}
    \item The map $\Kform{(\SWeN)}$ is injective.
    \item The map $\Kform{(\eSWeN)}$ is injective.
    \item The maps $\Rform{(\SWeN)}$, $\Rform{(\eSWeN)}$ are injective.
\end{enumerate}
\end{theorem}

\begin{proof}
    Claim (2) follows from claim (1) by applying $\eN$ to both sides.  Claim (3) then follows because each of $\DN\eN$, $\eN\DN\eN$ are free over $\cR$, 
    hence the kernel of each map must be both torsion and torsion free, hence zero.

    It remains to prove claim (1). For the remainder of the proof we work over base ring $\K$.  By their constructions, for each $\chi\in \spec(\SymY)$ we have natural isomorphisms, 
\[\DN\cdot \eN \rt{\SymY} \chi \cong \Ind^{\DN}_{\SH}(\chi \boxtimes \sgn),\quad \textrm{ and }\quad  \HKuniv\rt{\SymY}\chi \cong \HKchi.\]
By Proposition \ref{prop:Y and c}, $\SWeN$ descends to a $\DN$-module homomorphism,
\[
\SWeN\ot_{\iBt}\chi:\Ind^{\DN}_{\SH}(\chi \bt \sgn) \to F_N(\HKchi),
\]
such that $(\SWeN\ot_{\iBt}\chi)(1 \ot \uu) = z \otimes 1$.  This expression is nonzero by Corollary \ref{cor:HKchi}.  
Recall that $\DN \cdot \eN$ is a free right $\SymY$-module.   Hence
$\SWeN$ is injective if, and only if, its specialisation $\SWeN\ot_{\iBt}\chi$ is injective for generic $\chi$.  For generic $\chi$, $\Ind^{\DN}_{\SH}(\chi \bt \sgn)$ is irreducible and hence $\SWe\ot_{\iBt}\chi$, being nonzero, must be injective. Thus $\SWeN$ is injective.
\end{proof}

\begin{theorem}\label{thm:SWsurjective}
We have the following:
\begin{enumerate}
\item The map $\qform{(\eSWeN)}$ is an isomorphism.
    \item The map $\Rform{(\eSWeN)}$ is surjective.
    \item The map $\Kform{(\eSWeN)}$ is surjective.
\end{enumerate}

\end{theorem}
\begin{proof}
By Proposition \ref{prop:plusforms} and the right exactness of localization and quotients it suffices to prove these statements at the level of the polynomial forms introduced in Section \ref{sec:positive HK}. Claim (2) then follows from Claim (1) by an application of Nakayama's lemma, noting that the graded components of the modules are finitely generated $\cR$-modules. Claim (3) then follows from Claim (2) by the exactness of localization.

Thus it remains to prove Claim (1).  First note that the source $\qform{(\eN \DN^+ \eN)}$ is commutative, and identifies with the space of diagonally invariant polynomials 
$\kappa[\cX^+,\cY^+]^{\SN} := \kappa[\bar X_1, \ldots,  \bar X_N, Y_1, \ldots , Y_N]^{\SN}$, via the  
map which takes an invariant polynomial $p$ to $\eN \cdot p \cdot \eN$.
On the other hand, the target $\qform{(F_N(\HKuniv^+))}$ is also commutative, and identifies with
$(\cO(Mat_N)/I)^G = \cO(Mat_N)^G/I^G$,
where $I$ is the ideal generated by $\qform{C(\ad)}$. According to \cite{Gan-Ginzburg}[Theorem 1.2.1], the ideal $I^G$ is reduced, and this algebra identifies with the coordinate ring $\cO(\Comm_N)^G$ of the variety of pairs commuting matrices modulo simultaneous conjugation.

Putting these observations together, we may identify $\qform{\eSWeN^+}$
with a certain algebra map
\[
s: \kappa[\cX^+,\cY^+]^{\SN} \to \cO(\Comm_N)^G.
\]
Let us unpack the definition of the map in this setting. First, recall that given a pair of commuting matrices $\bar A,B \in \End(V)$, there is a natural algebra map
\begin{align*}
\kappa[\cX^+,\cY^+] &\to \End(V^{\otimes N})\\
p &\mapsto p_{\bar A,B}
\end{align*}
defined by the property that $(\bar X_i)_{\bar A,B}$ (respectively $(Y_i)_{\bar A,B}$) acts by $\bar A$ (respectively $B$) in the $i$th tensor factor. When $p \in \kappa[ \cX^+,\cY^+]^{\SN}$ is an invariant polynomial, $p_{ \bar A,B}$ acts by a scalar on the antisymmetric tensors $\wedgeN{N} 
\hookrightarrow V^{\otimes N}$. This scalar is the value of $s(p)$ on the pair of matrices $( \bar A,B) \in \Comm_{N}$. 

On the other hand, by \cite{Haiman-diagonal}[Proposition 6.2.1], the natural map induced by restriction to diagonal matrices induces an isomorphism
\[
r:\cO(\Comm_N)^G \to \cO(\ft \times \ft)^{\SN} \cong \kappa[\cX^+,\cY^+]^{\SN}.
\]
We claim that the the maps $r$ and $s$ are inverse to one another, and thus $s$ is an isomorphism as desired. 
Indeed, unwinding the definitions, the claim boils down to the fact that if $ \bar A=diag(a_1, \ldots , a_N)$ and $B= diag(b_1, \ldots , b_N)$ are diagonal matrices, then 
\[
s(p)( \bar A,B) = p(a_1, \ldots a_N,b_1, \ldots , b_N),
\]
which can be readily checked using the above description of the map $s$.
\end{proof}

In particular, combining Theorem \ref{thm:SWsurjective} with Theorem \ref{thm:SWinj}, we obtain that the maps $\Rform{(\eSWeN)}$ and $\Kform{(\eSWeN)}$ are isomorphisms, thus proving Theorem \ref{mainthm:HKuniv}.

\renewcommand{\Dist}{\operatorname{Dist}}

\section{Endomorphisms of $\HKe$ via shift isomorphism
} \label{sec:End HK via shift}

Throughout this section we fix $G=\GLN$, and set $n=N$, 
and we will assume $q=t^{-2}$.   We therefore adopt the shorthand from Notation \ref{not:parameters-all-together}. We also work over a base field $\K$. We sometimes abbreviate $\Y$ for $\K[\Y]$ and refer to $\Y$-weight spaces.  We discuss the modification of statements and their proofs for $\SLN$ at the end of this section. 

Sections \ref{sec:End HK via shift}  and \ref{sec:End via intertwiners} contain proofs and computations in the DAHA that will be used to prove the main theorems of this paper, via the functors $\ee F_N$ and $\F_N$.

Let us first recall the following basic categorical fact (see e.g. \cite[Appendix A]{G}).
\begin{proposition}\label{prop:barrbeck}
Let $\cC$ be a presentable abelian category and $P$ a compact projective object of $\cC$. Let $R=\End_\cC(P)^{\op}$. Then the natural functor 
\[F = \Hom_\cC(P,-): \cC \to R\modu\]
has a fully faithful left adjoint $F^L$, whose essential image is full subcategory of $\cC$ generated by $P$. In particular, given any quotient object $P\to Q$ we have an isomorphism
\[
\End_\cC(Q) \cong \End_{R}(F(Q)).
\]
\end{proposition}
\begin{remark}\label{rem:adjoint}
The left adjoint $F^L$ is very explicit. Given an $R$-module $L$, choose a presentation $L = \coker(R^{\oplus I} \to R^{\oplus J})$. Then $F^R(M)$ is given by $\coker (P^{\oplus I} \to P^{\oplus J})$ (recall that $R=\End(P)^{\op}$, so the morphisms in the two cases are given by the same data). It follows that if $M = \coker (P^{\oplus I} \to P^{\oplus J})$ is in the subcategory generated by $P$, then the counit map is an isomorphism:
\[
F^LF(M) \cong M.
\]
\end{remark}

Applying Proposition \ref{prop:barrbeck} to the category Theorem \ref{mainthm:HKuniv}.

\begin{corollary}\label{cor:Dq to HN}
    The functor
    \[
     \ee F_N: \Dqstr{G} \to \ee \HH_N \ee\modu
    \]
    has a fully faithful left adjoint $(eF_N)^L$ whose essential image is the subcategory generated by $\HKuniv$. In particular, given any quotient $M$ of $\HKuniv$, we have an isomorphism of algebras
    \[
    \End_{\Dqstr{G}}(M) \cong \End_{\ee \HH_N \ee}(\ee F_N(M)).
    \]
\end{corollary}
\begin{remark}
    It is proved in \cite{GJVY} that when $G=\GLN$, the functor $\ee F_N$ is conservative (as well as the functor $F_N$ itself) , and thus an equivalence. This is not true in the $SL_N$ case. 
\end{remark}
\subsection{The shift isomorphism} \label{sec:shift}
The shift isomorphism \cite{Marshall1999,Cherednik1995} is an isomorphism between the antispherical DAHA and the spherical DAHA with ``shifted'' parameters:\footnote{More precisely, following \cite{Cherednik1995} and \cite{Marshall1999}, the authors of \cite{Gorsky-Kivinen-Oblomkov} prove a shift isomorphism relating the anti-spherical double affine Hecke algebra to the spherical double affine Hecke algebra, upon trigonometric degeneration.  However the proofs apply verbatim to the non-degenerate setting.}
\[
\ee \HGN \ee \cong \ee^+ \HH_N^{\GL}({q,t q^{1/2}})\ee^+
\]
in our conventions. We will need this isomorphism at our specialization $q=t^{-2}$. Namely, taking the field $\K$ and $\DAHA{N}{N}$ as in Section \ref{sec:EndHK}, we obtain an isomorphism of $\K$-algebras
\[
\ee \DN \ee \cong \ee^+ \SmashDq \ee^+ = \Coinvq.
\]

The algebra $\cD_q(H)$ of $q$-difference operators on the maximal torus $H$ of $G$ is presented here as the group algebra of the doubled weight lattice $\Lambda \oplus \Lambda$ of $G$, with multiplication twisted by the symplectic pairing $\omega$ canonically attached to the symmetric Cartan pairing on $\Lambda$. 
Further, the shift isomorphism sends $\SymY$ to $\SymY$, where we have identified 
$i_\cB(\cO_q({H})) $ or $\K[0 \oplus \Lambda]$  with $\K[\Y]$.  Below we will identify $i_\cA(\cO_q({H}))$ or or $\K[\Lambda \oplus 0]$ with $\K[\X]$.
(It is important we express characters of $\SymY$ in terms of the loop parameter $q$ and not the quadratic parameter since this parameter changes under the shift.) 

Let us also recall that there is a Morita equivalence between $\Coinvq$ and $\SmashDq$. Indeed, this follows from Theorem 2.4 of \cite{Montgomery1980}, noting that $\DqqG{H}$ 
is simple (when $q$ is not a root of unity as in our case). Consider the composite functor
 \[
\Upsilon: \Dqstr{G} \xrightarrow[\ee F_N]{ \ } \ee\DN\ee\modu  \xrightarrow[\mathrm{shift}]{\sim} \Coinvq\modu \xrightarrow[\mathrm{Morita}]{\sim} \SmashDq\modu
\] 
Let us record the consequence of all of these observations, which will be used to compute the endomorphism algebras of Hotta-Kashiwara modules. 
\begin{corollary}\label{cor:Dq to smash}
    The functor
    \[
     \Upsilon: \Dqstr{G} \to \SmashDq\modu
    \]
    has a fully faithful left adjoint $\Upsilon^L$ whose essential image is the subcategory generated by $\HKuniv$. In particular, given any quotient $M$ of $\HKuniv$, we have an isomorphism of algebras
    \[
    \End_{\Dqstr{G}}(M) \cong \End_{\SmashDq}(\Upsilon(M)).
    \]
\end{corollary}

In particular, applying Corollary \ref{cor:Dq to smash} to $\HKuniv$ itself, we have an isomorphism of algebras (thus giving a $\q$-deformed Levasseur-Stafford isomorphism as in (\ref{eq:LS2})): 

\begin{equation}\label{eq:LS2}
\End_{\Dqstr{G}} (\HKuniv) 
\cong \Coinvq.
\end{equation}

\begin{remark}
We expect that there exists an analogous functor $\Upsilon$ and an isomorphism (\ref{eq:LS2}) for an arbitrary connected reductive group $G$, even though there is no analogue of the functor $F_N$. 
\end{remark}

In the remainder of this paper we will present two different proofs of each of Theorem \ref{mainthm:Springer} and Theorem \ref{mainthm:Springerchi} using Corollary \ref{cor:Dq to HN} and Corollary \ref{cor:Dq to smash} respectively. 

\subsection{Proof of Theorem \ref{mainthm:Springer} via the shift isomorphism} \label{subsec:End HK via shift}

To lighten notation in this section,  write
\begin{align*}
    &\smashD = \smashD_N = \SmashDq \text{  as in Corollary \ref{cor:Dq to smash},} 
    \\
    &\coinv =\coinv_N = \Coinvq \text{ as in equation \eqref{eq:LS2}} \\
    &\smashY = \smashY_N =\SmashY \subseteq \smashD \\
    &\SymSN = \SymS = \SymY \otimes \K[\SN]  \subseteq \smashY,
\end{align*}
where we have identified $i_\cB(\OqG{H}) $  with $\K[\Y]$.

Proposition \ref{prop:Y and c}  and Corollary \ref{cor:epsilon as qrho} give the following corollary.

\begin{corollary}\label{cor:eqrho}
\begin{enumerate}
    \item 
We have isomorphisms,
\[
\ee F_N(\HKe)\cong \ee \Ind_\SH^\DAHA{N}{N}(\epsilon \boxtimes \sgn) \cong \ee\Ind_{\Y}^{\DAHA{N}{N}}(\qrho)
\cong \Ind_{\SymY}^{\ee \DN \ee} \{ \qrho\},
\]
where $\qrho(Y_i) := q^{i-1} = t^{2-2i}$, or  $\qrho = \trho$
corresponds to \[\atup =(1, \cdots, q^{N-1})= (t^0,t^{-2},\ldots, t^{2-2N}).\]

\item 
More generally, if $\chi$ corresponds to $\atup$ via \eqref{eq:SymB} and is transverse, i.e., $i \neq j \implies a_i \neq a_j$, then $\ee F_N(\HKchi) \cong\Ind_{\SymY}^{\ee \DN \ee} \{ \atup\}$.
\end{enumerate}
\end{corollary}

We are now ready to state the main result of this section. 
The results from Section \ref{sec:EndHK} and Corollary \ref{cor:eqrho} imply that 
\begin{align} \label{eq:End HK to End eDe}
 \End(\HKe) \cong
    \End_{\ee \DN \ee}(\Ind_{\SymY}^{\ee \DN \ee} \{\qrho\}) , 
\end{align} 
which we show in the proposition below is isomorphic to  the group algebra of the symmetric group.

\begin{proposition} \label{prop:End smash}
    $\End_{\ee \DN \ee}(\Ind_{\SymY}^{\ee \DN \ee} \{\qrho\}) \cong \K[\SN]^{\op}$
\end{proposition}
\begin{proof}

 By Theorem \ref{thm:Ind SH}, for descending $\atup$ we have:
\begin{gather*}
    \ee \Ind_{\Y}^{\DN} \atup \cong \ee \Ind_{\SH}  \asgn = \Ind_{\SymY}^{\ee \DN \ee} \{\atup\},
\end{gather*}
and likewise a similar statement holds replacing $\ee$ with $\ee^+$ if the reverse $w_0(\atup)$ is descending. 
Note that for any $\sigma \in \SN$, and in particular $\sigma = w_0$, we have that 
$\Ind_{\Y}^{\smashY} \atup \cong \Ind_{\Y}^{\smashY} \sigma(\atup)$. 
Further for any affine permutation $w \in \AffSymn{N}$ we have $\Ind_{\Y}^{\smashD} \atup \cong \Ind_{\Y}^{\smashD} w(\atup)$
(recall the analogous statements for $\HY$ and $\DN$ are false, as seen in Example \ref{ex:non ssl}). 
In particular we have  an isomorphism,
\begin{align} \label{eq:1N}
  \Ind_{\Y}^{\smashD} w_0(\qrho)  \cong \Ind_{\Y}^{\smashD} 1^N  
\end{align}
where we have written $1^N$ for the $\K[\Y]$-module corresponding to $\atup = (1, 1, \cdots, 1)$ and will write  $\{ 1^N\}$ for $\{\atup\}$. 
It is easy to check the  module in \eqref{eq:1N} is $\Y$-semisimple, with each weight space of dimension $N!$.

Using the shift isomorphism, Morita equivalence, and \eqref{eq:1N} we  compute
\begin{align*}
    \End_{\ee \DN \ee}(\Ind_{\SymY}^{\ee \DN \ee} \{\qrho\})
    & \cong 
    \End_{\coinv}(\Ind_{\SymY}^{\coinv} \{\qrho\}) \\
    &\cong 
    \End_{\smashD}(\Ind_{\Y}^{\smashD} w_0(\qrho)) 
     \cong 
    \End_{\smashD}(\Ind_{\Y}^{\smashD} 1^N).
\end{align*}
Note 
$\Ind_{\Y}^{\smashY} 1^N \cong 1^N \boxtimes \K[\SN]$ is the module on which $\SN$ acts via the regular representation and all operators $(Y_i-1)$ vanish, for $1 \le i \le N$.  We now observe that the $1^N$ weight space of $\Ind_{\Y}^{\smashD} 1^N = \Ind_{\smashY}^{\smashD} \Ind_{\Y}^{\smashY}1^N$ is exactly $1^N \boxtimes \K[\SN]$.
To continue, the above endomorphism algebra is:
\begin{align*}
   \End_{\smashD}(\Ind_{\smashY}^{\smashD} 1^N  \boxtimes \K[\SN] )
    & \cong 
    \Hom_{\smashY}(1^N  \boxtimes \K[\SN] , \Res \Ind_{\smashY}^{\smashD} 1^N \boxtimes \K[\SN])
    \\
    & \cong 
    \Hom_{\smashY}(1^N  \boxtimes \K[\SN], 1^N \boxtimes \K[\SN])
    \\
    & \cong 
    \Hom_{\K[\SN]}( \K[\SN], \K[\SN]) \cong \K[\SN]^{\op}.
\end{align*}
\end{proof}
Proposition \ref{prop:End smash} with \eqref{eq:End HK to End eDe} complete the  proof of Theorem  \ref{mainthm:Springer}. 
\begin{remark}\label{rem:Ind Y not ssl}
  It is important to observe in the proof above  that we cannot apply Theorem \ref{thm:Ind SH}  to $\Ind_{\Y}^{\smashD} 1^N$ as $\fij{i}{j} \mid_{1^N} = 0 $ at quadratic parameter $1$. Indeed the   module in \eqref{eq:1N} is not isomorphic to  $\Ind_{\SymY \otimes \SN}^{\smashY} \{ 1^N \} \boxtimes \trivial$. It is easy to check the latter module is not $\Y$-semisimple; and neither is $\Ind_{\SymY}^{\Y} \{1^N\}$.   
\end{remark}

\subsection{Proof of Theorem \ref{mainthm:decomp} via the shift isomorphism.}\label{sec:proof of decomp}  In order to prove  the $M_\lambda$ appearing 
in the statement of Theorem \ref{mainthm:decomp} are distinct indecomposable modules,
we first prove  the irreducibility of the indecomposable summands of $\Ind_{Y}^{\smashD} 1^N$ given by the minimal idempotents of $\K[\SN]^{\op}$. 
\begin{proposition} \label{prop:Springer-decomp}
    As a $\smashD$-$\K[\SN]$ bimodule, we have a decomposition
    $$ \Ind_{\Y}^{\smashD} 1^N \cong \bigoplus_{\lambda} L_\lambda \otimes S^\lambda
    $$
    where the direct sum is indexed by partitions $\lambda$ of $N$ and $S^\lambda$ is the corresponding irreducible for $\SN$.
    Furthermore $L_\lambda \cong \Ind_{\smashY} 1^N  \boxtimes  S^\lambda$ is a simple $\smashD$-module.  

    Consequently, as a $\ee \DN \ee$-$\K[\SN]$ bimodule, we have a decomposition
    $$ \Ind_{\SymY}^{\ee \DN \ee} \{ \qrho \} \cong \bigoplus_\lambda \bar L_\lambda \otimes S^\lambda
    $$
    where $\bar L_\lambda$ is a simple $\ee \DN \ee$-module.
\end{proposition}

\begin{proof}
As noted in the proof of Proposition \ref{prop:End smash}, 
$\Ind_{\Y}^{\smashD} 1^N = \Ind_{\smashY} 1^N \boxtimes \K[\SN]$ and so by decomposing the regular representation, we see it  has  summands  of the form
$$L_\lambda := \Ind_{\smashY} 1^N  \boxtimes  S^\lambda,$$
which occur with multiplicity $\dim S^\lambda = | \syt(\lambda)|.$
Indeed, if $\{v_T \mid T \in \syt(\lambda) \}$ is a basis of $S^\lambda$ indexed by standard Young tableaux, then $L_\lambda$ has $\Y$-weight basis given by $\{X^\beta \otimes v_T \mid \beta \in\Z^N, T \in \syt(\lambda)\}$.  In particular, the $\Y$-weight of $X^\beta \otimes v_T$ is $q^\beta$.  (Recall the relation $Y_i X_j = q^{\delta_{ij}} X_j Y_i$ for $G=\GLN$.) 
Any simple submodule of $L_\lambda$ must contain some nonzero weight vector. If it is of weight $q^\beta$ then it must have the form $X^\beta \otimes v$ for some nonzero $v \in S^\lambda$.  Then the submodule also contains $1 \otimes v = X^{-\beta} X^\beta \otimes v$ and hence is all of $L_\lambda$.  

Next,
via the Morita equivalence between $\smashD$ and $\coinv$, we have a similar decomposition of $\ee^+ \Ind_{\Y}^{\smashD} 1^N \cong \Ind_{\SymY}^{\coinv} \{\qrho\}$ as a $\coinv$-module.  Then via the shift isomorphism,
an analogous  decomposition holds for  the anti-spherical DAHA module $\Ind_{\SymY}^{\ee \DN \ee} \{ \qrho \} =  \ee \Ind_{\Y}^{\DN} \qrho$.
\end{proof}

\begin{remark}$\SLGL$ For $G=\SLN$,  one modifies the relations on the quantum torus $\SmashDq$ appropriately and replaces $1^N$ with $(Z^{1/N})^N$. This completes the proof of Theorem \ref{mainthm:decomp}.
\end{remark}

Applying the left adjoint $\Upsilon^L$ (or respectively $(eF_N)^L$) to the first (respectively, second) isomorphism in Proposition \ref{prop:Springer-decomp}, we obtain a corresponding decomposition (see Remark \ref{rem:adjoint}: 
\[
\Upsilon^L(\Upsilon \HKchi) \cong \HKchi \cong \bigoplus_\lambda \Upsilon^L(L_\lambda) \otimes S^\lambda.
\]
As $\Upsilon^L$ is fully faithful and additive, the $\Dq$-modules $M_\lambda := \Upsilon^L(L_\lambda)$ are indecomposable\footnote{As mentioned in the introduction, the results of \cite{GJVY} allow us to upgrade ``indecomposable'' to ``simple''. Indeed, we will show that the functor $\Upsilon$ is an equivalence for $G=GL_N$ and a projection onto a direct summand for $G=SL_N$. \label{foot:irred}} and pairwise non-isomorphic for distinct $\lambda$. This proves Theorem \ref{mainthm:decomp}.

\subsection{Proof of Theorem \ref{mainthm:Springerchi} via the shift isomorphism.} \label{sec:proof Springerchi}

 As in the case of $\epsilon$, we get a generalization of Theorem \ref{mainthm:decomp} for transverse $\chi$, using notation from Theorem \ref{mainthm:Springerchi}. 
 We found it more intuitive to first give the proof of the special case $\chi  = \epsilon$ and then introduce appropriate modification for transverse $\chi$.  
 The necessity for transversality is explained in Remark \ref{rem:Ind Y not ssl}. We give an example of non-transverse $\chi$ in Example \ref{ex:nilpotent end}.
 
 \begin{theorem} \label{thm:decomp-chi}
 Suppose that $\chi$ is in transverse position, i.e.,  that $\Stab(\chi) \cap \SN = \{ \id\} $. Then
\begin{align*} 
\End_{\ee\DN\ee}(\ee F_N(\HKchi)) &\cong \K[W_J]^{\op} \, \text{ and } \\
\ee F_N(\HKchi) &\cong  \bigoplus_{\lambdatup} L_\lambdatup \boxtimes S^\lambdatup,
\end{align*}
where $ W_J = \sigma^{-1} \Stab(\chi) \sigma \subseteq \SN$ is a standard parabolic subgroup for some $\sigma \in \AffSym$. The decomposition is as   $\ee \DN \ee$-$W_J$ bimodules, and
the direct sum is indexed by multi-partitions $\lambdatup$ of total size $N$ and ``shape" $J$,  and $S^{\lambdatup}$ is the corresponding irreducible for $W_J$.  $L_{\lambdatup}$ is an irreducible $\ee \DN \ee$-module. 
\end{theorem}
\begin{proof}
Using  \eqref{eq:SymB}, we identify $\chi$ with $\{\atup\}$, where we choose $\atup \in \K^N$ to be descending and have the further hypothesis  that $\atup$ is transverse, i.e. $a_i \neq a_j$ for $i \neq j$.  
Let $\btup = \sigma^{-1} \atup$ so that $\Stab(\btup) = W_J \subseteq \SN$. 
 We have $\Ind_\Y^\smashD \atup \cong \Ind_\Y^{\smashD} \btup$. 
The  transversality of $\atup$ ensures $\Ind_{\SymS}^{\smashY} \{\atup\} \boxtimes \trivial \cong \Ind_\Y^{\smashY} \atup \cong \Ind_\Y^{\smashY} \btup $.
(As in Remark \ref{rem:Ind Y not ssl}, we warn the reader that $\Ind_{\SymS}^{\smashY} \{\btup\} \boxtimes \trivial \not\cong \Ind_\Y^{\smashY} \btup$ (unless $W_J$ is trivial).)
Similar to the proof of Proposition \ref{prop:End smash}  above, which was the special case $\btup = 1^N$, the $\btup$ weight space of $\Ind_{\Y}^{\smashD} \btup$ is exactly $\btup \boxtimes \K[W_J]$, where $\K[W_J]$ denotes the regular representation of the parabolic subgroup $W_J$. Then the analogous computation of the $\Hom$ space yields that the endomorphism algebra in question is $\K[W_J]^{\op}$.

Note $W_J = \Sm{\eta_1} \times \Sm{\eta_2} \times \cdots \times \Sm{\eta_\ell} \subseteq \SN$ for a corresponding composition $(\eta_1, \dots, \eta_{\ell})$ of $N$ which we  say has shape $J$. Thus we see the irreducible representations of $W_J$ are multipartitions also of shape $J$, i.e. $\lambdatup = (\lambda^{(1)}, \dots, \lambda^{(\ell)})$ where $\lambda^{(i)}$ is a partition of size $\eta_i$. 

    $W_J$ has corresponding parabolic subalgebra $\smashY_J \subseteq \smashY$, generated by $W_J$ and $\Y$.  Then $\Ind_{\smashY_J}^{\smashY} \btup \boxtimes S^{\lambdatup}$ is an irreducible representation of $\smashY$.  The indecomposable summands of $\Ind_{\Y}^{\smashD} \btup$ are 
    $\Ind_{\smashY}^{\smashD} (\Ind_{\smashY_J}^{\smashY} \btup \boxtimes S^{\lambdatup}) =\Ind_{\smashY_J}^{\smashD} \btup \boxtimes S^{\lambdatup} $. 
    A nonzero simple submodule would have to contain a $\Y$-weight vector, say of weight $(q^{\beta_1} b_{u(1)}, \cdots, q^{\beta_N} b_{u(N)})$ for some $\beta \in \ZZ^N$ and $u \in \SN$. ( $\SLGL$ Again we write this for $G=\GLN$ and one modifies appropriately for $G=\SLN$.)  Recall $b_i \neq b_j \implies b_i/b_j \neq q^z$ for any $z \in \ZZ$.  This weight vector must then have the from $X^\beta \otimes v$ for some $v \in S^{\lambdatup}$. As before, the submodule then contains $X^{-\beta} X^\beta \otimes v = 1 \otimes v$ and hence is the whole module.
    Thus the corresponding decomposition into indecomposables given by the first part of Theorem \ref{thm:decomp-chi} is actually a decomposition into simples.
\end{proof}

As in Section \ref{sec:proof of decomp}, this proves Theorem \ref{mainthm:Springerchi} by applying the left adjoint $\Upsilon^L$.

\section{Endomorphisms of $\HKe$ via intertwiners} \label{sec:End via intertwiners}
The goal of this section is to give an alternate proof of Proposition \ref{prop:End smash}, and hence of Theorem \ref{mainthm:Springer}, avoiding use of the shift isomorphism and instead relying on the Morita equivalence of Proposition \ref{prop:morita} below. 
Throughout this section we work over  base ring  $\K$.
We work directly in the DAHA instead of the anti-spherical DAHA, which allows us to make more explicit calculations and to use intertwiners. 
Using intertwiners we construct  explicit endomorphisms $\Phi_w$ of $F_N(\HKe)$, defining a ring homomorphism,

\begin{align} \label{eq:End homom}
\Phi:\K[\Sym_N]^{op}&\to \End_{\DN}(F_N(\HKe))  \\      
w &\mapsto \Phi_w,
\end{align}
and we proceed to show it is an isomorphism.  
To connect the results in this section to Theorem \ref{mainthm:Springer}, we will require the following algebraic input:

\begin{proposition}[\cite{GJVY}]\label{prop:morita}
The sign idempotent $\eN$ defines a Morita equivalence between $\DAHA{N}{N}$ and $\eN\DAHA{N}{N}\eN$.
\end{proposition}

Analogous results to Proposition \ref{prop:morita} are known for the rational degeneration of $\DAHA{N}{N}$ (see e.g. \cite{BEG}), however we could not find a reference in the literature for this statement for $\DAHA{N}{N}$.  In our forthcoming paper \cite{GJVY}, we give a complete proof of a more general version of this statement.

\begin{remark}\label{rem:Morita}
The reader will note that Proposition \ref{prop:morita} could be applied to many of the results of Section \ref{sec:EndHK} to replace $\eN\DAHA{N}{N}\eN$ with $\DAHA{N}{N}$ we have done in Proposition \ref{prop:HKgen}.  We have avoided doing so to ensure that the present paper, with the exception of the present section, which merely gives an alternative and more constructive proof of Theorem \ref{mainthm:Springer} -- is independent of \cite{GJVY}. 
\end{remark}

\begin{proposition}\label{prop:HKgen} 
For any strongly equivariant $\Dq$-module $M$ in the subcategory generated by $\HKuniv$ we have  isomorphisms,
\[
\End_{\Dqstr{G}}(M) \cong \End_{\Dq}(M) \cong \End_{\eN \DAHA{N}{N} \eN}( \eN F_N(M))
 \cong \End_{\DAHA{N}{N}}(F_N(M)).
\]

\end{proposition}

\begin{proof}
The first isomorphism is due to Proposition \ref{prop:full subcategory}.

For the second isomorphism, recall that the universal Hotta-Kashiwara module $\HKuniv$ is a projective object in $\Dqstr{G}$ by Proposition \ref{prop: Hom HK gives invariants}, hence we have an equivalence of categories between the subcategory generated by $\HKuniv$ and $\End(\HKuniv)^{op}\modu,$ given by applying the functor $\ee\cdot F_N \cong \Hom(\HKuniv,-)$.   The last isomorphism follows from the Morita equivalence asserted in Proposition \ref{prop:morita}.

\end{proof}

In the case $M=\HKchi$, we can describe the rightmost endomorphism algebra of Proposition  \ref{prop:HKgen}, by first realising $F_N(\HKchi)$ as an induced module, using the identification \eqref{eq:SymB}.

We can strengthen  Corollary \ref{cor:eqrho} using Proposition \ref{prop:morita}.
\begin{proposition}\label{prop:End FN HK}
    There is an isomorphism of $\DN$-modules
    \[
    F_N(\HKchi) \cong 
    \Ind^{\DN}_{\SH}(\chi \bt \sgn)
    \]
\end{proposition}
\begin{proof}
Recall from the proof of Theorem \ref{thm:SWinj} that there is a natural map
\[
\SWeN\ot_{\iBt}\chi:\Ind^{\DN}_{\SH}(\chi \bt \sgn) \to F_N(\HKchi).
\]
By Proposition \ref{prop:morita}, to check that $\SWeN\ot_{\iBt}\chi$ is an isomorphism, it suffices to check that it is an isomorphism after multiplying on the left by $\eN$. But this follows from the Theorem \ref{thm:SWsurjective}, tensoring the isomorphism  
\[
\eSWeN: \eN \DN \eN \to \HKuniv^G 
\]
with the character $\chi$. 
\end{proof}

Proposition \ref{prop:Y and c}, together with Proposition \ref{prop:End FN HK}, imply the following corollary, which is a strengthening of the first part of Corollary \ref{cor:eqrho}.

\begin{corollary}\label{cor:qrho}
We have isomorphisms,
\[
F_N(\HKe)\cong \Ind_\SH^\DAHA{N}{N}(\epsilon \boxtimes \sgn) \cong \Ind_{\Y}^{\DAHA{N}{N}}(\qrho).
\]

\end{corollary}

\begin{proof} 
The first isomorphism is the case $\chi=\epsilon$ of Proposition \ref{prop:End FN HK}.  
Recall from Notation \ref{not:parameters-all-together} we have $t=\q$ and $q=t^{-2}$.   
The latter isomorphism is Corollary \ref{cor:epsilon as qrho} combined with Theorem \ref{thm:Ind SH}, as $\trho$ is descending. 
\end{proof}

Using Proposition \ref{prop:morita}, we can strengthen 
Proposition \ref{prop:Springer-decomp} to give a decomposition of $\Ind_{\Y}^{\DN} \qrho$ into irreducibles.
\begin{corollary}\label{cor: Ind decomp}
As a $\DN_N$-$\K[\SN]$ bimodule, we have a decomposition
    $$F_N(\HKe) \cong  \Ind_{\Y}^{\DN} \qrho \cong \bigoplus_{\lambda} L_\lambda^q \otimes S^\lambda
    $$
    where the direct sum is indexed by partitions $\lambda$ of $N$ and $S^\lambda$ is the corresponding irreducible for $\SN$.
    Furthermore $L_\lambda^q \cong \Ind_{\HY}   S^\lambda_t$ is a simple $\DN$-module. 
\end{corollary}

\begin{proof}
    The finite Hecke algebra $\Hf$ has irreducible modules labeled by partitions, and we
    can inflate these to the affine Hecke algebra $\HY$ along the evaluation homomorphism that sends $T_i \mapsto T_i$, $Y_1 \mapsto 1$.  We call the resulting irreducible AHA-module $S_t^\lambda$.

    We observe $\ee L_\lambda^q = \bar L_\lambda$ from Proposition \ref{prop:Springer-decomp}.  Given $\ee$ yields a Morita equivalence, the result follows.
\end{proof}

As mentioned in Footnote \ref{foot:irred}, from this corollary,  the functor $\Upsilon$ and Proposition \ref{prop:morita}, we may conclude the $M_\lambda$ of Theorem \ref{mainthm:decomp} are irreducible, not just indecomposable.

\begin{remark} \label{rem:JW}
Given a $\Dq$-module $M$, recall that while the action of $\DAHA{N}{n}$ is only well defined on $(V^{\ot n}\ot M)^G$, the action of $\HY$ is well defined on all of $V^{\ot n}\ot M$.  Let us record the following peculiar corollary of \cite[Theorem 1.3]{JW} (which we will not use in the remainder of the paper):  for any $n$, the operator $Y=Y_1$ acting on $V^{\ot n}\otimes 1 \subset V^{\ot n} \otimes \HKe$ satisfies the characteristic equation,
\[
(Y-1)(Y-\q^{-2})\cdots(Y-\q^{(2-2N)}) = 0
\]
We note that when $\q=1$ this becomes the condition that $Y$ lies in the unipotent cone, however for quantum parameter $\q$ generic it implies $Y$ acts diagonalizably. 
\end{remark}

Now we proceed to build the DAHA machinery to show $\Phi$ of \eqref{eq:End homom} is a well-defined homomorphism.   Although redundant to give this second proof of Proposition \ref{prop:End smash}, it is more direct than working with the anti-spherical and spherical DAHA, and we build some worthwhile DAHA tools along the way.

\subsection{Poles and zeros of intertwiners}  Our strategy is to define the required endomorphisms $\Phi_w$ in terms of renormalized intertwiners $\nuj{i}\in\HHloc$ from Definition \ref{def:intertwiners}.  The expressions $\fij{i}{i+1}^{\pm 1}$ can introduce poles and zeroes into the definition, and hence much of the technical complication in making sense of the $\Phi_w$ is to keep track of these.

To this end, let us begin by collecting some simple observations concerning zeros of the $\fij{i}{j}$ when evaluated at the character $\qrho = \trho$ appearing in Corollary \ref{cor:qrho}. We have:
\[(Y_i-Y_j)|_{\trho} = 0 \quad \textrm{if, and only if, $j-i\equiv 0 \mod (n+1)$}.
\]
With the $q=t^{-2}$ specialization we observe:
\[\fij{i}{j} = tY_i - t^{-1} Y_{j} =  t^{-1}(Y_{i+n}-Y_j) = t(Y_i-Y_{j-n}),\]
hence we conclude:
\[\fij{i}{j}|_{q^\rho} =0 \quad \textrm{ if, and only if $j-i \equiv n \mod (n+1)$}.\]

Let us first rewrite
\[
\nui = T_i \frac{t \fij{i-n}{i+1}}{\fij{i}{i+1}} + \frac{(t-t^{-1})Y_{i+1}}{\fij{i}{i+1}}.
\]
As a consequence of this formula and Equation \eqref{eqn:phivsnu}, if we consider the normal ordering $\normal{\nu_w}$ of some operator $\nu_w$ as described in Section \ref{sec:normal} (see Definition \ref{def:normal}) below, we see that the poles and zeros of $\normal{\nu_w}$ are controlled by the set of inversions of $w$.  This is captured in the following Definition, as will become clearer in the proof of Theorem \ref{thm:nopoles}.

\begin{definition}
Fix $w\in\AffSym$.  We say that an inversion $(i,j)\in\Inv(w)$ is \defterm{vanishing} if $(j-i) \equiv 0 \mod (n+1)$, and \defterm{singular} if $(j-i)\equiv n \mod (n+1)$.  We denote by $\Inv_0(w)$ and $\Inv_\infty(w)$, respectively, the sets of vanishing and singular   inversions.
\end{definition}

Recall that $\Sn\subset\AffSym$ is the stabilizer of $(0,\ldots,0)$ with respect to the action on $\Z^n$.  Let us denote by $\trrho\in\AffSym$ the translation by the element $-\rho =(0,-1,\ldots,1-n)$, so that $\trrho \cdot\rho = (0,\ldots,0)$. 
Recall the symbol $\cdot$ denotes our usual (not dot) action from Section \ref{sec:intertwiners}.
Clearly, the stabilizer of $\rho$ with respect to the above action of $\AffSym$ on $\Z^n$ is
\[
    \operatorname{Stab}_{\AffSym}(\rho) 
    = \trrho^{-1} \Sym_n \trrho.
\]
We work with $-\rho$ in this section since $\qrho = t^{-2 \rho}$ but conventionally $\AffSym$ acts  on $\ZZ^n$, not $(2\ZZ)^n$. 
The stabilizer is therefore generated by the elements
\begin{align}
    \trrho^{-1} s_i \trrho & = [1,\,2,\,\cdots, \,\, i\!+1\!+n, \,\, i\!-\!n ,\,\,\cdots, \, n] \notag \\
    &= \underbrace{s_is_{i+1}\cdots s_{n-1}s_{i-1}\cdots s_2s_1}_{\deltai^{-1}} s_0 \underbrace{s_1s_2\cdots s_{i-1}s_{n-1}\cdots s_{i+1}s_i}_{\deltai}, \label{eq:ellen one}
\end{align}
expressed first in ``window notation" (see \cite{bjorner-brenti}), and then as a reduced expression.  In particular we record the equation $\ellen(\trrho^{-1}s_i\trrho) = 2n-1$.
Note that $\delta_i = [2,\,3,\, \cdots,\, i,\, n,\, 1, \, \, i\!+\!1,\, \cdots,\, n\!-\!1]$  is an $n$-cycle.

The following lemma is straightforward; we omit the proof, but instead give an example in Figure \ref{ex:bijection}

\begin{lemma}\label{lem:inversions}
Let $w\in\Sym_n\subset\AffSym$.  We have bijections,
\[\begin{array}{rl}\alpha:& \Inv(w) \to \Inv_{0}(\trrho^{-1}w\trrho)\\&(i,j) \mapsto (i,j+(j-i)n)\end{array},\qquad    \begin{array}{rl}\beta:&\Inv(w) \to \Inv_{\infty}(\trrho^{-1}w\trrho),\\
&(i,j) \mapsto (i,j + (j-i+1)n)\end{array}.
\]
\end{lemma}

 \subsection{Normal orderings} \label{sec:normal}
 Recall that $\HH_n(N)$ has a vector space decomposition $\HH_n(N) = \HX \otimes \K[\Y]$, which extends to a decomposition $\HHloc = \HX\otimes \K[\Yloc]$.  Let us fix the standard basis of $\HX$ consisting of elements $T_w$, for $w\in \AffSym$. 
\begin{definition} \label{def:normal}
Given $h\in\HHloc$, we will write $\normal{h}\in\HX\otimes \K[\Yloc]$ for its \defterm{normal ordering},
    \[
    \normal{h} = \sum_{w\in\AffSym} T_wg_w,
    \]
    where $g_w\in \K[\Yloc]$ and all but finitely many $g_w$ are nonzero.  In particular, we have $h\in\HH$ if, and only if $g_w\in\K[\Y]\subset \K[\Yloc]$ for all $w\in\AffSym$.
\end{definition}
    
Given a normally ordered element $\normal{h}= \sum_x T_xg_x$ of $\HHloc$, we will say it has a leading term if there exists $w \in \AffSym$ such that $g_w \neq 0$, and for $x \neq w$,  $g_x\neq 0$ implies $\ellen(x)<\ellen(w)$.  In this case we call $T_w$ the \defterm{leading term} and $g_w\in\K[\Yloc]$ the \defterm{leading coefficient}.  We stress that, because length does not define a total ordering, not every normally ordered expression has a leading term.

We now turn to the key technical result of this section.

\begin{theorem}\label{thm:nopoles}
Let $w\in\Sym_n$, and consider the normal ordering of $\nuj{\trrho^{-1}w\trrho}$,
\[
\normal{\nuj{\trrho^{-1}w\trrho}} = \sum_{x\in \AffSym}T_xg_x.
\]
Then $\normal{\nuj{\trrho^{-1}w\trrho}}$ has leading term $T_{\trrho^{-1}w\trrho}$, and the leading coefficient $g_{\trrho^{-1}w\trrho}$ has neither a zero nor a pole at $\trho = \qrho$.  Moreover, no lower order coefficient $g_x$ has a pole at $\trho$.
\end{theorem}

\begin{proof}
To simplify notation, we will write 
\begin{gather}\label{eqn:define aij bij for nu}
\aij{i}{j}=\frac{t\fij{i-n}{j}}{\fij{i}{j}}, \qquad \bij{i}{j} = \frac{(t-t^{-1})Y_{j}}{\fij{i}{j}}, \quad \textrm{ hence } \nui = T_i\aij{i}{i+1} + \bij{i}{i+1}.
\end{gather}
Note that $\aij{i}{j} = \aij{i+n}{j+n}$ and $\bij{i}{j} = \bij{i+n}{j+n}$.

Pick a reduced  expression $\trrho^{-1}w\trrho=s_{i_1}\cdots s_{i_p}$, thereby inducing an ordering
\[
\Inv(\trrho^{-1}w\trrho) = \{(i_1,j_1)\ldots (i_p,j_p)\}
\]
on the set of inversions of $\trrho^{-1}w\trrho$.  
We have:

\[
\nuj{\trrho^{-1}w\trrho} = \sum_{\underline{\epsilon}\in \{0,1\}^p} T_{i_1}^{\epsilon_1}\ldots T_{i_p}^{\epsilon_p}\aij{i_1}{j_1}^{\epsilon_1}\bij{i_1}{j_1}^{(1-\epsilon_1)}\cdots \aij{i_p}{j_p}^{\epsilon_p}\bij{i_p}{j_p}^{(1-\epsilon_p)}.\]
It is clear from the above sum that $\normal{\nuj{\trrho^{-1}w\trrho}}$ has leading term $T_{\trrho^{-1}w\trrho}$.

By inspection, we note that at $\trho$, each $\aij{i}{j}$ contributes a zero precisely when $(i,j)$ is a vanishing inversion of $\trrho^{-1}w\trrho$, while each $\aij{i}{j}$ and each $\bij{i}{j}$ contributes a pole precisely when $(i,j)$ is a singular inversion of $\trrho^{-1}w\trrho$. 
Below we analyze the cancellation of such poles once one simplifies the products above.

It follows from the bijection between singular and vanishing inversions established in Lemma \ref{lem:inversions} that the leading coefficient $g_{\trrho^{-1}w\trrho}$ has neither a zero nor a pole at $\trho$. 
More precisely, we note that the term $T_{\trrho^{-1}w\trrho}$ 
has $\epsilon_r = 1$ for all vanishing 
$(i_r,j_r) \in \Inv_{0}(\trrho^{-1}w\trrho)$. By Lemma \ref{lem:inversions} the potential zero at $a_{i_r, j_r}$ will cancel with the potential pole of $a_{i_s, j_s}$ for 
$(i_s,j_s) = (i_r,j_r+n) \in \Inv_{\infty}(\trrho^{-1}w\trrho)$. 
Thus the leading coefficient $g_{\trrho^{-1}w\trrho}$, once simplified, has neither zero nor pole at $\trho$.

It remains to show that no pole arises at $\trho$ in any coefficient $g_x$, with $\ellen(x)<\ellen(w)$. 
For this, we require a more elaborate inductive argument to ensure pole cancellation for the sum of the $2^{p-1}$ terms with $\epsilon_r=0$.  We will induct on $\ellen(w)$;  if $\ellen(w)=0$ there is nothing to prove.

Next consider $\ellen(w)=1$. We give a more careful analysis of the terms for which $\epsilon_r=0$ when $w=s_i$.  Using the reduced expression from \eqref{eq:ellen one}, note $\trrho^{-1} s_i \trrho$ has one vanishing inversion $(i_r, j_r) = (i, i+1+n) \in \Inv_0(\trrho^{-1} s_i \trrho)$ for $r=n$ and one singular inversion $(i_s, j_s) = (i, i+1+2n) \in \Inv_\infty(\trrho^{-1} s_i \trrho)$ for $s=1$ corresponding to $\Inv(s_i) = \{ (i, i+1)\}.$  Thus the only potential poles arise from 
$b_{i_s, j_s} = b_{i-n, i+1+n}$ or $a_{i_s, j_s} = a_{i-n, i+1+n}$.
We first rewrite 
\begin{eqnarray*}
    \nuj{\trrho^{-1} s_i \trrho} &=
    T_i \nuj{s_i \deltai^{-1}} T_0 \nuj{\deltai} a_{i-n , i+1} a_{i-n , i+1+n} +
     \nuj{s_i \deltai^{-1}} T_0 \nuj{\deltai}  a_{i-n , i+1} b_{i-n , i+1+n} \\
     &+ \, 
    T_i \nuj{s_i \deltai^{-1}}  \nuj{\deltai}  b_{i-n , i+1} a_{i-n , i+1+n} +
     \nuj{s_i \deltai^{-1}}  \nuj{\deltai}   b_{i-n , i+1} b_{i-n , i+1+n} \\
      &=
    T_i \nuj{s_i \deltai^{-1}} T_0 \nuj{\deltai} a_{i-n , i+1} a_{i-n , i+1+n} +
     \nuj{s_i \deltai^{-1}} T_0 \nuj{\deltai}  a_{i-n , i+1} b_{i-n , i+1+n} \\
     & \quad + \, 
     ((t^2-1)T_i (Y_i + t^2 Y_{i+1}) + (t^{-1} Y_i + (t-t^{-1} - t^3) Y_{i+1}) f_{i, i+1}^{-1}  b_{i-n , i+1}.
\end{eqnarray*}

Above we simplified  the last two terms which we see have no pole at $\trho$.
In the following computations, let $\Bullet$ denote either the symbol $a$ or $b$.
For the first two terms, recall 
$a_{i_r, j_r} \,  \Bullet_{i_s,j_s} = a_{i-n , i+1} \, \Bullet_{i-n, i+1+n}$ has no pole at $\trho$.
However, we need to check the terms that contribute to $\epsilon_n=0$ arising from moving $\Bullet_{k,n}$ or $\Bullet_{1,k}$ past $T_0$. 
The following identities
\begin{align*}
    \nuj{s_i \deltai^{-1}} &= \sum  T_{i+1}^{\epsilon_2}
    \cdots T_1^{\epsilon_{n-1}} a_{i+1 , n}^{\epsilon_2} b_{i+1 , n}^{1-\epsilon_2}\cdots a_{1,2}^{\epsilon_{n-1}} b_{1,2}^{1-\epsilon_{n-1}}
    \\
    a_{1,k} T_0 &= T_0 a_{0,k} + b_{0,k} b_{1,k} Y_1/Y_k & 
    \text{for $1 < k \le i$}
    \\
     b_{1,k} T_0 &= T_0 b_{0,k} -t b_{0,k} b_{1,k} Y_1/Y_k & 
    \text{for $1 < k \le i$}
    \\
     a_{k,n} T_0 &= T_0 a_{k, n+1} + b_{k , n+1} b_{k,n} Y_k/Y_n & 
    \text{for $i+1 \le k <n$}
    \\
     b_{k,n} T_0 &= T_0 b_{k , n+1} -t b_{k, n+1} b_{k,n} Y_k/Y_n & 
    \text{for $i+1 \le k <n$}
    \\
    \Bullet_{0,k}\,  \Bullet_{1,k}\,  \nuj{\deltai} & = \nuj{\deltai}\,  \Bullet_{i-n , k-1}\,  \Bullet_{i+1 , k-1} & 
    \text{for $1 < k \le i$}
    \\
     \Bullet_{k,n} \, \Bullet_{k , n+1}\,  \nuj{\deltai}  &= \nuj{\deltai} \, \Bullet_{k+1 , i}\,  \Bullet_{k+1 , i+1+n} & 
    \text{for $i+1 \le k <n$}
    \\
    \nuj{\deltai} &= \sum  T_{1}^{\epsilon_{n+1}}
    \cdots T_i^{\epsilon_{2n-1}} a_{1 , i+1}^{\epsilon_{n+1}} b_{1 , i+1}^{1-\epsilon_{n+1}}\cdots a_{i , i+1}^{\epsilon_{2n-1}} b_{i , i+1}^{1-\epsilon_{2n-1}}
\end{align*}
show that when simplifying the first two terms $\normal{T_i \nuj{s_i \deltai^{-1}} T_0 \nuj{\deltai} a_{i-n , i+1} a_{i-n , i+1+n}}$ and $\normal{\nuj{s_i \deltai^{-1}} T_0 \nuj{\deltai} a_{i-n , i+1} b_{i-n , i+1+n}}$
have no poles at $\trho$.

We have shown $\normal{ \nuj{\trrho^{-1} s_i \trrho}}$ has no poles, i.e., is well-defined when evaluated at $\trho$.  This is equivalent to the following:
 if $\vv$ is a weight vector of weight $\trho$ so is $\vv' = \normal{\nuj{\trrho^{-1} s_i \trrho}} \vv$. 
 (The first part of the proof shows $\vv' $ is nonzero, although that fact is not necessary here.)
Now we can proceed with the induction. 
Given $w$ with $\ellen(w) > 1$, choose $i$ so $\ellen(w s_i) < \ellen(w)$ and write $u = w s_i$ which has smaller length than $w$.   Next $\normal{\nuj{\trrho^{-1} w \trrho}} \vv = \normal{\nuj{\trrho^{-1} u \trrho}} \, \normal{\nuj{\trrho^{-1} s_i \trrho}} \vv = \normal{\nuj{\trrho^{-1} u \trrho}} \vv'$, which is a well-defined weight vector of weight $\trho$ by induction.  
Thus the coefficients of $\normal{\nuj{\trrho^{-1} w \trrho}}$ also have no pole at $\trho$. 

\end{proof}

 \begin{example} \label{ex:bijection}
 Let $n=N=3$.  Then $\trrho = [1,-1,-3]$ in window notation.
  Consider $w = s_1 s_2 = [2,3,1] \in \Sm{3}$  
  pictured as    \begin{tikzpicture}
\begin{scope}[shift={(0,0)}, scale=0.2]
 \draw[blue!45!white, line width = 2pt] (1,0)--(5,3);
  \draw[blue!45!white, line width = 2pt] (5,0)--(9,3);
   \draw[blue!45!white, line width = 2pt] (9,0)--(1,3);
\end{scope}
\end{tikzpicture}, which has inversions $\{ (1,3), (2,3)\}$
 Then    $u:=\trrho^{-1} w \trrho = [1,5,9] \, [2,3,1] \, [1,-1,-3] = [5,6,-5]$
 is determined by $u(1), u(2+3)$ and $u(3+6)$, which must be permuted among themselves to stabilize $\qrho$.   
 For instance $Y_1 \nuj{u} = \nuj{u} Y_9 = \nuj{u} q^{-2} Y_3 = \nuj{u} t^4 Y_3$  and we note $Y_1 \mid_{\trho} = t^4 Y_3 \mid_{\trho} = t^0$.

 We mark below $i$ how $Y_i$ acts on a weight vector of weight $\qrho = \trho$.
 
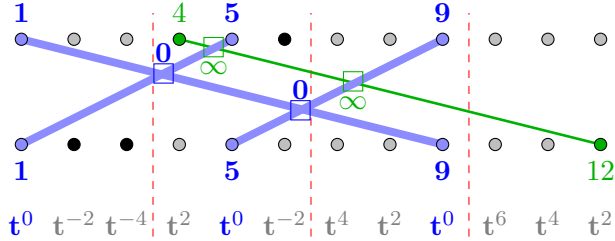
\begin{center}
\begin{tikzpicture}
   \begin{scope}[shift={(0,0)}, scale=0.7]
   
     \draw[red, dashed, thin] (3.5,-1.75)--(3.5,2.5);
\draw[red, dashed, thin] (6.5,-1.75)--(6.5,2.5);
  \draw[red, dashed, thin] (9.5,-1.75)--(9.5,2.5);

\node[blue] at (1,2.5) {$\mathbf 1$};
\node[blue] at (5,2.5) {$\mathbf 5$};
\node[blue] at (9,2.5) {$\mathbf 9$};
\node[blue] at (1,-.5) {$\mathbf 1$};
\node[blue] at (5,-.5) {$\mathbf 5$};
\node[blue] at (9,-.5) {$\mathbf 9$};

\node[green!70!black] at (4,2.5) {$4$};
\node[green!70!black] at (12,-.5) {$12$};

\node[blue] at (1,-1.5) {$\mathbf t^0$};
\node[gray] at (2,-1.5) {$\mathbf t^{-2}$};
\node[gray] at (3,-1.5) {$\mathbf t^{-4}$};
\node[gray] at (4,-1.5) {$\mathbf t^{2}$};
\node[blue] at (5,-1.5) {$\mathbf t^0$};
\node[gray] at (6,-1.5) {$\mathbf t^{-2}$};
\node[gray] at (7,-1.5) {$\mathbf t^{4}$};
\node[gray] at (8,-1.5) {$\mathbf t^{2}$};
\node[blue] at (9,-1.5) {$\mathbf t^0$};
\node[gray] at (10,-1.5) {$\mathbf t^{6}$};
\node[gray] at (11,-1.5) {$\mathbf t^{4}$};
\node[gray] at (12,-1.5) {$\mathbf t^{2}$};

\draw[green!70!black, line width = 1pt] (12,0)--(4,2);

 \draw[blue!45!white, line width = 3pt] (1,0)--(5,2);
  \draw[blue!45!white, line width = 3pt] (5,0)--(9,2);
   \draw[blue!45!white, line width = 3pt] (9,0)--(1,2);
   
\node[rectangle,draw=blue] (r) at (3.7,1.35) {};
\node[rectangle,draw=blue] (s) at (6.3,0.65) {};
\node[blue] at (3.7,1.75) {$\mathbf 0$};
\node[blue] at (6.3,1.05) {$\mathbf 0$};

\node[rectangle,draw=green!70!black] (t) at (7.3,1.20) {};
\node[rectangle,draw=green!70!black] (u) at (4.65,1.85) {};
\node[green!70!black] at (7.3,0.80) {$\infty$};
\node[green!70!black] at (4.65,1.45) {$\infty$};

    \draw[fill=blue!45!white] (1,0) circle (.7ex);
    \draw[fill=black] (2,0) circle (.7ex);
    \draw[fill=black] (3,0) circle (.7ex);

        \draw[fill=gray!50!white] (4,0) circle (.7ex);
    \draw[fill=blue!45!white] (5,0) circle (.7ex);
          \draw[fill=gray!50!white] (6,0) circle (.7ex);
    \draw[fill=gray!50!white] (7,0) circle (.7ex);
    \draw[fill=gray!50!white] (8,0) circle (.7ex);
    \draw[fill=blue!45!white] (9,0) circle (.7ex);

        \draw[fill=gray!50!white] (10,0) circle (.7ex);
    \draw[fill=gray!50!white] (11,0) circle (.7ex);
    \draw[fill=green!70!black] (12,0) circle (.7ex);

    \draw[fill=blue!45!white] (1,2) circle (.7ex);
    \draw[fill=gray!50!white] (2,2) circle (.7ex);
    \draw[fill=gray!50!white] (3,2) circle (.7ex);

        \draw[fill=green!70!black] (4,2) circle (.7ex);
    \draw[fill=blue!45!white] (5,2) circle (.7ex);
          \draw[fill=black] (6,2) circle (.7ex);
    \draw[fill=gray!50!white] (7,2) circle (.7ex);
    \draw[fill=gray!50!white] (8,2) circle (.7ex);
    \draw[fill=blue!45!white] (9,2) circle (.7ex);

        \draw[fill=gray!50!white] (10,2) circle (.7ex);
    \draw[fill=gray!50!white] (11,2) circle (.7ex);
    \draw[fill=gray!50!white] (12,2) circle (.7ex);

    \end{scope}
\end{tikzpicture}
\end{center} 

\noindent We depict $u$ above. The {\color{blue} blue} boxed  inversions labeled {\color{blue} $0$}   are  $\Inv_0(u) = \{(1,9) , \, (2,6) \equiv (5,9)  \}= \{\alpha(1,3), \, \alpha(2,3)\}$ and correspond to where {\color{blue} blue} strands cross.  These are in bijection with $\Inv_\infty(u) = \{(1,12) , \, (2,9) \equiv (5,12) \}= \{\beta(1,3), \, \beta(2,3)\}$ corresponding to the  forced inversions from the $N$-translated {\color{green!70!black} green} strand.  These are labeled {\color{green!70!black} $\infty$} and {\color{green!70!black} green} boxed.
Below we draw more of $u \in \AffSymn{3}$, taking note that contributions to $\Inv_\infty(u)$ only occur for ${\color{blue} i } < {\color{green!70!black} j}$. Observe $\ellen(u)=8$ and 
$\Inv(u) = \{ {\color{green!70!black} (1,12)}, (2,12),
{\color{blue} (1,9)}, {\color{green!70!black} (2,9)},
 (1,6),  {\color{blue} (2,6)},
 (1,3), (2,3)  \}$ 
  which is ordered corresponding to the reduced expression $u= s_1 s_2 s_0 s_1 s_2 s_0 s_1 s_2$.

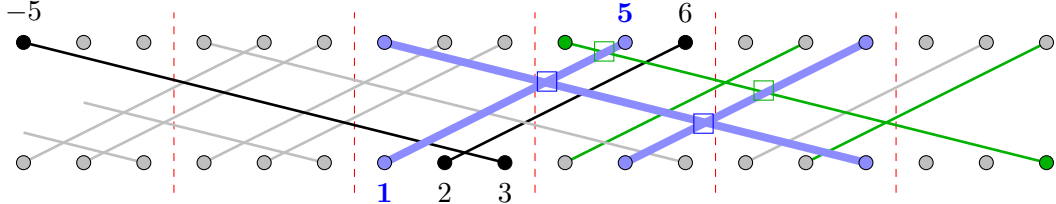
\begin{tikzpicture}
    \begin{scope}[shift={(0,0)}, scale=0.8]
   
    \draw[red, dashed, thin] (-2.5,-.5)--(-2.5,2.5);
      \draw[red, dashed, thin] (.5,-.5)--(.5,2.5);
        \draw[red, dashed, thin] (3.5,-.5)--(3.5,2.5);
          \draw[red, dashed, thin] (6.5,-.5)--(6.5,2.5);
            \draw[red, dashed, thin] (9.5,-.5)--(9.5,2.5);

\node[blue] at (5,2.5) {$\mathbf 5$};
\node[blue] at (1,-.5) {$\mathbf 1$};

\node[black] at (2,-.5) {$2$};
\node[black] at (3,-.5) {$3$};
\node[black] at (6,2.5) {$6$};
\node[black] at (-5,2.5) {$-5$};

\draw[gray!50!white, line width = 1pt] (-5,0)--(-1,2); 
\draw[gray!50!white, line width = 1pt] (-4,0)--(0,2);
\draw[gray!50!white, line width = 1pt] (-3,0)--(-5,.5);

\draw[gray!50!white, line width = 1pt] (-2,0)--(2,2);
\draw[gray!50!white, line width = 1pt] (-1,0)--(3,2);
\draw[gray!50!white, line width = 1pt] (0,0)--(-4,1);
%
\draw[black, line width = 1pt] (2,0)--(6,2);
\draw[black, line width = 1pt] (3,0)--(-5,2);
\draw[green!70!black, line width = 1pt] (4,0)--(8,2);
\draw[gray!50!white, line width = 1pt] (6,0)--(-2,2);
\draw[gray!50!white, line width = 1pt] (7,0)--(11,2);
\draw[green!70!black, line width = 1pt] (8,0)--(12,2);

\draw[green!70!black, line width = 1pt] (12,0)--(4,2);

 \draw[blue!45!white, line width = 3pt] (1,0)--(5,2);
  \draw[blue!45!white, line width = 3pt] (5,0)--(9,2);
   \draw[blue!45!white, line width = 3pt] (9,0)--(1,2);

\node[rectangle,draw=blue] (r) at (3.7,1.35) {};
\node[rectangle,draw=blue] (s) at (6.3,0.65) {};

\node[rectangle,draw=green!70!black] (t) at (7.3,1.20) {};
\node[rectangle,draw=green!70!black] (u) at (4.65,1.85) {};

    \draw[fill=gray!50!white] (-5,0) circle (.7ex); 
    \draw[fill=gray!50!white] (-4,0) circle (.7ex);
    \draw[fill=gray!50!white] (-3,0) circle (.7ex);
    \draw[fill=gray!50!white] (-2,0) circle (.7ex);
    \draw[fill=gray!50!white] (-1,0) circle (.7ex);

    \draw[fill=gray!50!white] (0,0) circle (.7ex);
    \draw[fill=blue!45!white] (1,0) circle (.7ex);
    \draw[fill=black] (2,0) circle (.7ex);
    \draw[fill=black] (3,0) circle (.7ex);

        \draw[fill=gray!50!white] (4,0) circle (.7ex);
    \draw[fill=blue!45!white] (5,0) circle (.7ex);
          \draw[fill=gray!50!white] (6,0) circle (.7ex);
    \draw[fill=gray!50!white] (7,0) circle (.7ex);
    \draw[fill=gray!50!white] (8,0) circle (.7ex);
    \draw[fill=blue!45!white] (9,0) circle (.7ex);

        \draw[fill=gray!50!white] (10,0) circle (.7ex);
    \draw[fill=gray!50!white] (11,0) circle (.7ex);
    \draw[fill=green!70!black] (12,0) circle (.7ex);

        \draw[fill=black] (-5,2) circle (.7ex);
    \draw[fill=gray!50!white] (-4,2) circle (.7ex);
    \draw[fill=gray!50!white] (-3,2) circle (.7ex);
    \draw[fill=gray!50!white] (-2,2) circle (.7ex);
    \draw[fill=gray!50!white] (-1,2) circle (.7ex);

    \draw[fill=gray!50!white] (0,2) circle (.7ex);
    \draw[fill=blue!45!white] (1,2) circle (.7ex);
    \draw[fill=gray!50!white] (2,2) circle (.7ex);
    \draw[fill=gray!50!white] (3,2) circle (.7ex);

        \draw[fill=green!70!black] (4,2) circle (.7ex);
    \draw[fill=blue!45!white] (5,2) circle (.7ex);
          \draw[fill=black] (6,2) circle (.7ex);
    \draw[fill=gray!50!white] (7,2) circle (.7ex);
    \draw[fill=gray!50!white] (8,2) circle (.7ex);
    \draw[fill=blue!45!white] (9,2) circle (.7ex);

        \draw[fill=gray!50!white] (10,2) circle (.7ex);
    \draw[fill=gray!50!white] (11,2) circle (.7ex);
    \draw[fill=gray!50!white] (12,2) circle (.7ex);

    \end{scope}
    \end{tikzpicture}

    \noindent We may expand $\nuj{u}  = \sum_{\underline{\epsilon}\in \{0,1\}^8} T_1^{\epsilon_1} T_2^{\epsilon_2}T_0^{\epsilon_3}T_1^{\epsilon_4}T_2^{\epsilon_5}T_0^{\epsilon_6}T_1^{\epsilon_7}T_2^{\epsilon_8} 
    \Bullet_{-2,9} \, \Bullet_{-1,9}\, \Bullet_{-2,6}\, \Bullet_{-1,6}\, \Bullet_{1,6}\, \Bullet_{-1,3}\, \Bullet_{1,3}\, \Bullet_{2,3}$
    $  = \sum_{\underline{\epsilon}\in \{0,1\}^8}   T_1^{\epsilon_1} T_2^{\epsilon_2}T_0^{\epsilon_3}T_1^{\epsilon_4}T_2^{\epsilon_5}T_0^{\epsilon_6}T_1^{\epsilon_7}T_2^{\epsilon_8} 
    \Bullet_{\color{green!70!black} 1,12} \, \Bullet_{2,12}\, \Bullet_{\color{blue} 1,9}\, \Bullet_{\color{green!70!black} 2,9}\, \Bullet_{1,6}\, \Bullet_{\color{blue} 2,6}\, \Bullet_{1,3}\, \Bullet_{2,3}$
 \end{example}
\subsection{Proof of Theorem \ref{mainthm:Springer} via intertwiners}
Let us now outline an alternative proof of Theorem \ref{mainthm:Springer} that relies on Proposition \ref{prop:morita}.

 In the next two theorems, we use $\vv$ to denote a generator  of a one-dimensional $\K[\Y]$-module.
Observe that if $\vv$ has $\Y$-weight $\trho = \qrho$ then for a $g \in \K[\Yloc]$ without zeros or poles at $\trho$ we have $\normal{g}|_{\trho} \ot \vv = \normal{g} \ot \vv$.

\begin{theorem} \label{thm:End qrho}
The map
\begin{align*}
\Phi: \K[\Sym_n]^{op} &\to \End_{\HH}(\Drho),\\
w &\mapsto \Phi_w,
\end{align*}
where $\Phi_w(h\ot\vv) = (h \normal{\nuj{\trrho^{-1}w\trrho}})\ot\vv$, for $ h \in \HH $ defines a $\K$-algebra isomorphism.
\end{theorem}

\begin{proof}
First, let us confirm that each $\Phi_w$ is indeed a $\HH$-linear endomorphism of $\Drho$.  Theorem \ref{thm:nopoles} ensures the well-definedness of each expression $(\normal{\nuj{\trrho^{-1}w\trrho}})\ot\vv$, while  the intertwiner property implies that $(\normal{\nuj{\trrho^{-1}w\trrho}})\ot\vv$ lies in the $\trho$ $\Y$-weight space. Hence Frobenius reciprocity for the functor $\Ind^\HH_\Y$ produces the asserted module-homomorphism.

The homomorphism property for $\Phi$ follows from the following identity in $\HHloc$:
\[\nuj{\trrho^{-1}w\trrho}\nuj{\trrho^{-1}u\trrho} = \nuj{\trrho^{-1}wu\trrho}.\]
Finally, we observe that, as a consequence of Theorem \ref{thm:nopoles}, the set $\{(\normal{\nuj{\trrho^{-1}w\trrho}})\ot\vv \mid w \in \Sn\}$ forms a basis of the $\trho$ $\Y$-weight space, hence $\Phi$ is an isomorphism.
\end{proof}

 Theorem \ref{mainthm:Springer} now follows, using Proposition  \ref{prop:HKgen} and Corollary \ref{cor:qrho}.

\subsection{Proof of Theorem \ref{mainthm:Springerchi} via intertwiners} \label{sec: chi neq qrho}

In this section, we consider the endomorphism algebra $\End_{\HH}(\Ind_\Y^\HH \atup)$ for more general  $\atup$ than $\trho = \qrho$, but still  descending. 
Note $\qrho = \trho$ is descending.
In other words, by Theorem \ref{thm:Ind SH}, we consider $\End_{\HH}(\Ind_{\SH}^\HH \asgn)$.

It is convenient to introduce the following terminology:
\begin{definition} Let $a,b, r\in\K^\times$.  We say that $a$ and $b$ are \defterm{in the same $r$-line} if $a/b=r^{z}$ for some $z\in\mathbb{Z}$, and otherwise that they are \defterm{in distinct $r$-lines}.
\end{definition}

\begin{theorem} \label{thm:End atup}
Let $\atup \in (\K^\times)^n$ be descending
such that entries in the same $t^2$-line are in consecutive position.
Further assume
 that $\atup$ is \defterm{transverse}, meaning that $\Stab_{\AffSym}(\atup) \cap \Sn =\Stab_{\Sn}(\atup)= \{ \id \}$, i.e., that $a_i \neq a_j$ if $i \neq j$. Let $\gamma_{\atup} \in \AffSym$ be of minimal length such that $\gamma_{\atup} \Stab_{\AffSym}(\atup) \gamma_{\atup}^{-1} = W_J \subseteq \Sn$ for appropriate standard parabolic subgroup.
Then the map
\begin{align*}
\Phi: \K[W_J]^{op} &\to \End_{\HH}(\Ind_\Y^\HH \atup),\\
w &\mapsto \Phi_w,
\end{align*}
where $\Phi_w(h\ot\vv) = (h \normal{ \nuj{{\gamma_\atup}^{-1} w \gamma_{\atup} }}  |_{\atup} )\ot\vv$, 
defines a $\K$-algebra isomorphism.
\end{theorem}

\begin{proof}
Let us outline how to modify the proof of Theorem \ref{thm:End qrho} to apply here.
Note that $J $ is determined by the different $t^2$-lines that the $a_i$ lie on. 
In particular, the generalized $\atup$ $\Y$-weight space of $\Ind_\Y^\HH \atup$ has dimension $|W_J|$. 
Similar to Lemma \ref{lem:inversions}, if $ w \in W_J$, we still have  bijections between $\Inv(w)$, the inversions of $\gamma_{\atup}^{-1} w \gamma_{\atup}$ that potentially introduce poles to the coefficients of $\normal{ \nuj{{\gamma_\atup}^{-1} w {\gamma_{\atup} }} }$, and those inversions that potentially introduce zeros. (The criterion for an inversion being vanishing or singular is more complicated in this case and not just determined by height.)
As before, we are able to construct $|W_J|$ linearly independent $\Y$-weight vectors of weight $\atup$ and corresponding endomorphisms to yield the isomorphism $\Phi$. 
\end{proof}

Theorem \ref{mainthm:Springerchi} now follows, using 
Proposition \ref{prop:HKgen}.

\begin{remark} \label{rem:SLGL affsym} $\SLGL$
Very few modifications need be made for the $G=\SLN$ case of the above proof. Wherever $\AffSym$ appears, we instead take its quotient by the subgroup generated by $\pi^n$. In computing the stabilizer of a weight we  use the modified action of $\Spi$ as in \eqref{eq:modify pi for SL}. 
We also compare entries in the same $t^{2/N}$-lines, and similarly we declare these to be descending if whenever $a_i/a_j = t^{2z/N }$ with $z \in \Z$  and $i<j$ then $z \ge 0$. The $\atup$ we will consider furthermore are required to have $\prod_{i=1}^N a_i = \Zprod$.  \end{remark}

\begin{remark}The non-transverse case, where $\Stab_{\AffSym}(\atup) \cap \Sn = W_K \neq \{ \id \}$, is more complicated. We may still write $\gamma_{\atup} \Stab_{\AffSym}(\atup) \gamma_{\atup}^{-1} = W_J \subseteq \Sn$, but we have $K \subseteq J$.
While the {\em generalized} $\atup$ $\Y$-weight space of $\Ind_\Y^\HH \atup$ still  has dimension $|W_J|$, the {\em ordinary}
$\atup$ $\Y$-weight space only has dimension $|W_{J \setminus K}|$.
Hence $\dim \End_{\HH}(\Ind_{\Y}^\HH \atup) = |W_{J \setminus K}|$ in this case. 
However the ring structure of $\End_{\HH}(\Ind_\Y^\HH \atup)$ is more complicated than that of the semisimple algebra $\K[W_{J \setminus K}]$, and can contain nilpotent elements.  This is related to the fact that the induced module is not semisimple.   See Example \ref{ex:nilpotent end} below.
\end{remark}

\begin{example}\label{ex:non ssl}
This example shows the necessity of $\atup$ being descending for various results in this paper.
Let $n=N=2$. Let $\atup = \trho = (t^0, t^{-2})$ and $\btup = (t^{-2}, t^0)$.   Then $\Ind_{\SH}^\HH \asgn \simeq \Ind_{\SH}^\HH \{\btup\}\boxtimes \sgn \simeq  \Ind_{\Y}^\HH \atup \simeq {\mathbf T} \oplus {\mathbf S}$. 
We have let ${\mathbf T} = \Ind_{\HY}^{\HH} {\mathtt{triv}}$ for ${\mathtt {triv}}$ the one-dimensional $\HY$-module on which $(T_1-t), \, (Y_1- t^{-2}), (Y_2 -1)$ all vanish. 
We let ${\mathbf S} = \Ind_{\HY}^{\HH} {\mathtt {sgn}}$ for ${\mathtt {sgn}}$ the one-dimensional $\HY$-module on which $(T_1 + t^{-1}), \, (Y_1- 1), (Y_2 -t^{-2})$ all vanish.  As in Theorem \ref{thm:End qrho}, $\End_{\HH}(\Drho) \simeq \End_{\HH}({\mathbf T} \oplus {\mathbf S}) \simeq \K[\Sm{2}]^{\op}$.

Next consider $\Ind_{\Y}^\HH \btup$. Observe $
\btup$ is not descending, but since it is in the same $\AffSymn{2}$-orbit as $\atup$, it will have the same composition factors as above.  In fact
$0 \to {\mathbf S} \to \Ind_{\Y}^\HH \btup \to {\mathbf T} \to 0$ is non-split and so the induced modules has trivial endomorphism algebra.  To see the above exact sequences does not split, one can compute that the $\btup$ $\Y$-weight space of ${\mathbf S}$ is zero. (This is a ``rectangular" representation from \cite{Jordan-Vazirani}.)  On the other hand, the generalized $\btup$ weight space of ${\mathbf T}$ is two-dimensional, but the ordinary weight space is just one-dimensional. 
In particular ${\mathbf T}$ is not $\Y$-semisimple.
(Note, we could have also taken $\btup = (t^0, t^0)$ as in Section \ref{subsec:End HK via shift}, keeping Remark \ref{rem:Ind Y not ssl} in mind.)
\end{example}
\begin{example}\label{ex:nilpotent end}
This example shows the necessity of $\atup$ being transverse for Theorem \ref{thm:End atup} to hold.  We give an example of an induced module that has a nonzero nilpotent endomorphism, hence neither it nor its endomorphism algebra is  semisimple.  Let $n=N=3$ and $\atup = (t^0, t^0, t^{-2})$ which is descending but clearly not transverse. We take $\gamma$ to be translation by $(0,0,-1)$, 
so the stabilizer of $\atup$ is 
\[\gamma^{-1}\Sm{3} \gamma = \{\id, \,  s_1, \,  s_2 s_1 s_0 s_1 s_2, \,  s_1 s_2 s_1 s_0 s_1 s_2,  \, s_2 s_1 s_0 s_1 s_2 s_1, \,  s_1 s_2 s_1 s_0 s_1 s_2 s_1\}.\]
 The first difference is that  $\nuj{s_1} \vv = \vv$ and $T_1 \vv$ is a generalized $\atup$ $\Y$-weight vector.  It is easy to see the generalized $\atup$ weight space has dimension 6 and the ordinary $\atup$ weight space has dimension $\le 3$.  But in fact, as asserted above, the weight space only has dimension 2, as indicated by  $|J \setminus K| = 1$.  Next, consider $\gamma^{-1} s_2 \gamma = s_2 s_1 s_0 s_1 s_2$. 
   One can expand 
   \begin{align*}   
  \nuj{\gamma^{-1} s_2 \gamma} &= 
   \sum_{\underline{\epsilon}\in \{0,1\}^5} T_{2}^{\epsilon_1}  T_{1}^{\epsilon_2}T_{0}^{\epsilon_3}T_{1}^{\epsilon_4}
   T_{2}^{\epsilon_5} \Bullet_{-1,6} \Bullet_{-1,1} \Bullet_{-1,3} \Bullet_{1,3} \Bullet_{2,3} \\
   &= 
   \sum_{\underline{\epsilon}\in \{0,1\}^5} T_{2}^{\epsilon_1}  T_{1}^{\epsilon_2}T_{0}^{\epsilon_3}T_{1}^{\epsilon_4}
   T_{2}^{\epsilon_5} \Bullet_{2,9} \Bullet_{3,1} \Bullet_{2,6} \Bullet_{1,3} \Bullet_{2,3}. \end{align*} 
   As in the proof of Theorem \ref{thm:nopoles}, for each term with $\epsilon_3=1$, the potential zero in $a_{-1,3}$ cancels the potential pole of $\Bullet_{-1,6}$.  The collection of terms with $\epsilon_3=0$ all combine and cancel the potential pole of $\Bullet_{-1,6}$. However the fact that $s_1$ stabilizes $\atup$ gives an unexpected pole at $\Bullet_{-1,1}$ that cannot be cancelled.  We rectify this by rescaling by $\fij{-1}{1}$, that is, we take $\nuj{s_2 s_1 s_0 s_1 s_2} \fij{-1}{1} = \nuj{s_2} \phij{s_1} \nuj{s_0 s_1 s_2}$ whose normal ordering has no poles at $\atup$ and has leading term corresponding to $\gamma^{-1} s_2 \gamma$.  In other words, $\normal{\nuj{\gamma^{-1} s_2 \gamma} \fij{-1}{1}} \vv $ in a nonzero $\Y$-weight vector of weight $\atup$.  However one may compute that the corresponding nonzero endomorphism determined by $\vv \mapsto \normal{\nuj{\gamma^{-1} s_2 \gamma} \fij{-1}{1}} \vv$ is nilpotent---it squares to 0. 
\end{example}

\printbibliography
\end{document}